\documentclass[a4paper,11pt]{amsart}
\input{Macros}
\usepackage[frenchb, english]{babel}
\usepackage{tabu}

\newcommand{\fm}{\mathfrak{m}}
\newcommand{\fn}{\mathfrak{n}}
\newcommand{\fp}{\mathfrak{p}}

\newcommand{\fr}{\mathfrak{r}}

\DeclareMathOperator{\Char}{char}

\DeclareMathOperator{\pr}{pr}

\begin{document}

\title{Gluing for stable families of surfaces in mixed characteristic}
\author{Quentin Posva}
\date{}

\email{quentin.posva@hhu.de}
\address{Mathematisches Institut der Heinriche-Heine-Universität Düsseldorf, Universitätsstr.1, 24.13.03.33}

\maketitle

\begin{quote}
\textsc{Abstract.} We study the normalizations of non-normal stable families of slc surfaces over an excellent DVR. In mixed characteristic, we establish a gluing statement that is relevant for the properness of the moduli space of such surfaces. We also study the fibers of such glued families, in mixed an equi-characteristic, and provide an essential result for slc adjunction in residue characteristic $>5$.

\bigskip
\noindent 
\emph{Keywords and MS classification: Stable surfaces, semi-log canonical singularities, mixed and positive characteristic. 14B05, 14G17, 14G45, 14J17, 14J29.}
\end{quote}

\tableofcontents

\section{Introduction}

Semi-log canonical (slc) singularities play a central role in the moduli theory of canonically polarized varieties in characteristic zero: they appear at the boundary of the compact moduli space of stable varieties \cite{Kollar_Families_of_varieties_of_general_type}. It is expected that slc varieties play a similar role in positive and mixed characteristics, at least for varieties of dimension $\leq 2$ (see \cite{Kollar_Families_of_stable_3folds_in_pos_char} for some issues appearing in higher dimensions).

Recent breakthroughs in the Minimal Model Program (MMP) for threefolds in mixed characteristic \cite{Bhatt&Co_MMP_for_3folds_in_mixed_char, Takamatsu_Yoshikawa_MMP_for_semistable_3folds_in_mixed_char} have led to advances in the moduli theory of stable surfaces in mixed characteristic. The corresponding moduli stack is known to be a separated Artin stack with finite diagonal and of finite type \cite[Theorem I]{Bhatt&Co_MMP_for_3folds_in_mixed_char}. Properness is not known yet, except for some specific subspaces \cite[Theorem J]{Bhatt&Co_MMP_for_3folds_in_mixed_char}. 

In characteristic zero, properness of the moduli stack of stable varieties is proved through a delicate use of semi-stable reductions and MMP methods: see \cite[\S 2.4]{Kollar_Families_of_varieties_of_general_type} for a presentation of the proof. In addition to semi-stable reduction and MMP, that proof needs a technical gluing statement. Indeed, we consider a pointed curve $(t\in T)$ together with a stable family $X^0\to T^0=T\setminus\{t\}$, and we need to complete the family over $T$, possibly after a finite base-change $T'\to T$. The variety $X^0$ is only demi-normal in general, and the MMP might fail for demi-normal varieties \cite[Example 5.4]{Fujino_Fundamental_Theorems_for_slc_pairs}. Thus we normalize $X^0$ and try to complete the family $(X^0)^n\to T^0$. If we succeed, we need to de-normalize the completed family to recover the correct generic fiber. In characteristic zero, this is achieved through Koll\'{a}r's gluing theory \cite[\S 5]{Kollar_Singularities_of_the_minimal_model_program}, that gives a dictionnary between slc stable varieties and their normalizations.

In \cite{Posva_Gluing_for_surfaces_and_threefolds} I have extended this dictionnary to surfaces and threefolds in positive characteristic, and given applications to the properness of the moduli space of stable surfaces in positive characteristic (see also \cite{Patakfalvi_Projectivity_moduli_space_of_surfaces_in_pos_char} for related results about this moduli space). In the present paper, I study the analogous gluing statement necessary for the proof of properness of the moduli stack of stable surfaces in mixed characteristic along the lines sketched above. The tailor-made statement is the following one:

\begin{theorem_intro}[see \autoref{proposition:gluing_exists} and \autoref{proposition:canonical_sheaf_descends}]\label{theorem:gluing_for_surface_families}
Let $R$ be an excellent DVR of mixed characteristic with maximal ideal $\pi R$. Then normalization gives a bijection
	\begin{equation*}
		\begin{pmatrix}
		\text{Threefold pairs } (X,\Delta)\\
		\text{flat and proper over }R\\
		\text{such that }(X,\Delta+X_\pi) \text{ is slc}\\
		\text{and }K_X+\Delta\text{ is ample}
		\end{pmatrix}
		\overset{1:1}{\longrightarrow}
		\begin{pmatrix}
		\text{Threefold pairs } (\bar{X},\bar{D}+\bar{\Delta})  \\
		\text{flat and proper over }R \\
		\text{such that }(\bar{X},\bar{D}+\bar{\Delta}+\bar{X}_\pi) \text{ is lc} \\
		\text{and }K_{\bar{X}}+\bar{D}+\bar{\Delta} \text{ is ample}\\
		\text{plus a generically fixed point free}\\
		R\text{-involution }\tau\text{ of }(\bar{D}^n,\Diff_{\bar{D}^n}\bar{\Delta}).
		\end{pmatrix}
	\end{equation*}
(On the right-hand side, $\bar{D}$ is understood to be a Weil divisor, possibly reducible, whose every coefficient is equal to $1$.)
\end{theorem_intro}

See \autoref{thm_intro:slc_adjunction} below and the discussion preceding it for some comments on the slc condition on $(X,\Delta+X_\pi)$. Let us say a word about the proof of \autoref{theorem:gluing_for_surface_families}. We construct an inverse map to the normalization process: $\bar{X}$ should be the normalization of an slc quotient $\bar{X}/R(\tau)$. It turns out that the main task is the construction of the quotient scheme, for then the slc structure is provided by work of Koll\'{a}r \cite[\S 5]{Kollar_Singularities_of_the_minimal_model_program} (see \autoref{proposition:canonical_sheaf_descends} for details).

To show that the quotient exists, thanks to a result of Witaszek \cite[Theorem 1.4]{Witaszek_Keel_theorem_and_quotients_in_mixed_char} and to Koll\'{a}r's theory \cite[Theorem 5.13]{Kollar_Singularities_of_the_minimal_model_program}, we only need to show that the equivalence relations induced by $\tau$ on the special fiber $X_\pi$ is finite. To achieve this, we relate $\bar{D}^n_\pi$ to a divisor contained in the round-down of $\Diff_{\bar{X}_\pi^n}(\bar{D}+\bar{\Delta})$, and show that the relation on $\bar{X}_\pi$ comes from a log involution on the reduced boundary of the $k(\pi)$-pair $(\bar{X}_\pi^n,\Diff_{X_\pi^n}(\bar{D}+\bar{\Delta})$. This is done using adjunction and classification of codimension two lc singularities. Then we are in position to apply the results from \cite{Posva_Gluing_for_surfaces_and_threefolds} and obtain finiteness. We also obtain a conditional finiteness result in case $\dim\bar{X}=4$: see \autoref{proposition:gluing_exists} for the precise statement.

We emphasize that all our results are independent of $\dim \bar{X}$ and of the properties of the positive characteristic residue field $k(\pi)$ --- except for \autoref{proposition:gluing_exists}, which shows that the quotient indeed exists in several situations.

\bigskip
In the second part of the paper, we study the fibers of $X=\bar{X}/R(\tau)$ over $\Spec R$ (which we now allow to be of equicharacteristic). The involution $\tau\in \Aut_R(\bar{D}^n)$ induces an involution on both fibers of $\bar{D}^n\to \Spec R$. We can therefore take the quotient of both fibers of $\bar{X}\to \Spec R$ by the induced equivalence relations. We ask whether these quotients are equal to the fibers of $X$. This commutativity property is always true for the generic fiber (\autoref{lemma:commutativity_for_generic_fiber}). For the special fiber, we first study the general case of the normalization of a demi-normal scheme over $\Spec R$ and give some sufficient conditions (see \autoref{proposition:equivalent_cond_for_commutativity} and \autoref{proposition:some_minimal_conditions_for_commutativity}). In particular, we prove that if $\bar{D}$ is normal and $\Char k(\pi)>2$ then the commutativity property always holds.

Then we turn to the case that interests us the most, when $(X,\Delta+X_\pi)$ is slc. We are able to refine the method of \autoref{proposition:some_minimal_conditions_for_commutativity} to prove that commutativity holds for families of slc surfaces under mild hypothesis:

\begin{theorem_intro}[\autoref{thm:commutativity_in_slc_case}]\label{thm_intro:commutativity}
Suppose $(X,\Delta+X_\pi)$ is slc of dimension $3$, flat and quasi-projective over $\Spec R$ and that $\Char k(\pi)>2$. Assume that $\bar{D}_\pi$ is reduced. Then the special fiber of $X\to \Spec R$ is the quotient of the special fiber of $\bar{X}\to \Spec R$ by the equivalence relation induced by $\tau\otimes_R k(\pi)$. 
\end{theorem_intro}

We remark that $\bar{D}_\pi$ is the scheme-theoretic intersection of two lc centers of $(\bar{X},\bar{D}+\bar{\Delta}+\bar{X}_\pi)$. It is a question of independent interest whether intersections of lc centers are reduced. This is true in characteristic zero \cite[Theorem 7.8]{Kollar_Singularities_of_the_minimal_model_program} (see also \cite[Corollary 3]{Kollar_Seminormal_log_center_and_deformations}), but not much is known in mixed and positive characteristics. Nonetheless, we can prove that the hypothesis of \autoref{thm_intro:commutativity} holds in relative dimension two and residue characteristic large enough:

\begin{theorem_intro}[\autoref{proposition:intersection_of_lc_centers}]\label{thm_intro_reducedness}
Let $(\bar{X},\bar{D}+\bar{\Delta}+\bar{X}_\pi)$ be a lc threefold pair of finite type over $\Spec R$. Assume that $\Char k(\pi)> 5$. Then $\bar{D}_\pi$ is reduced.
\end{theorem_intro}

The proof relies on the technique of \cite[Theorem 7.8]{Kollar_Singularities_of_the_minimal_model_program} and on the results of \cite{Bernasconi_Kollar_vanishing_theorems_for_threefolds_in_char_>5}. Putting together \autoref{thm_intro:commutativity} and \autoref{thm_intro_reducedness} we obtain:

\begin{theorem_intro}\label{thm_intro:commutativity_for_surfaces}
Suppose $(X,\Delta+X_\pi)$ is slc of dimension $3$, flat and quasi-projective over $\Spec R$ and that $\Char k(\pi)>5$. Then the fibers of $X$ are the quotients of the fibers of the normalization $\bar{X}$.
\end{theorem_intro}

A corollary of \autoref{thm_intro:commutativity_for_surfaces} is the fact that $X_\pi$ is $S_2$ if $\bar{X}_\pi$ is $S_2$: this follows from \autoref{lemma:quotient_of_S_2_is_S_2}. After the first version of this paper appeared, it was proved in \cite{Arvidsson_Bernasconi_Patakfalvi_Properness_of_mod_space} and \cite{Arvidsson_Vanishing_thm_for_bir_morphism_of_3folds} that if $(X,\Delta+X_\pi)$ is lc then $X_\pi$ is $S_2$ as soon as $\Char k(\pi)>5$. Combined with the above corollary, we obtain the following adjunction result (which settles the condition $\mathbf{(S_2)}$ of \cite[\S 6]{Posva_Gluing_for_surfaces_and_threefolds} and \cite[\S 4]{Arvidsson_Bernasconi_Patakfalvi_Properness_of_mod_space}):

\begin{theorem_intro}\label{thm_intro:slc_adjunction}
Let $(X,\Delta)$ be a flat quasi-projective threefold pair over $\Spec R$. Assume that $\Char k(\pi)>5$. Then $(X,\Delta+X_\pi)$ is slc in a neighbourhood of $X_\pi$ if and only if $(X_\pi, \Diff_{X_\pi}(\Delta))$ is slc.
\end{theorem_intro}
\begin{proof}
The results mentioned above imply that $X_\pi$ is demi-normal, see \cite[Proposition 3.2.1]{Arvidsson_Posva_Normality_of_mim_lc_centers}. Therefore we can perform adjunction along $X_\pi\hookrightarrow X$, and obtain a pair $(X_\pi, \Diff_{X_\pi}(\Delta))$ \cite[Definition 4.2]{Kollar_Singularities_of_the_minimal_model_program}. We have a commutative diagram
		$$\begin{tikzcd}
		(\bar{X}_\pi^n, \Diff(\bar{\Delta}+\bar{D}))\arrow[r]\arrow[d] & (\bar{X},\bar{D}+\bar{\Delta}+\bar{X}_\pi)\arrow[d] \\
		(X_\pi,\Diff(\Delta))\arrow[r] & (X,\Delta+X_\pi)
		\end{tikzcd}$$
where, by definition of the different $\Diff$, the vertical arrows are crepant birational morphisms. Thus $(X_\pi,\Diff(\Delta))$ is slc if and only if $(\bar{X}_\pi^n, \Diff(\bar{\Delta}+\bar{D}))$ is lc; and similarly $(X,\Delta+X_\pi)$ is lc along $X_\pi$ if and only if $(\bar{X}_\pi^n, \Diff(\bar{\Delta}+\bar{D}))$ is lc along $\bar{X}_\pi$. So we conclude by (inversion) of adjunction in the lc case, which holds by \cite[Lemma 3.3]{Patakfalvi_Projectivity_moduli_space_of_surfaces_in_pos_char} and \cite[Corollary 10.1]{Bhatt&Co_MMP_for_3folds_in_mixed_char}.
\end{proof}

Let me mention that the overall proof of \autoref{thm_intro:commutativity_for_surfaces} relies mostly on the birational geometry of \emph{normal} lc pairs: the additional input needed to obtain the slc case is of combinatorial nature. See \autoref{rmk:cond_for_comm_thm} for details.

To end this paper, we take the opportunity to study more generally the Serre properties of a demi-normal quotient. In view of \cite[10.18]{Kollar_Singularities_of_the_minimal_model_program}, it is not surprising that there is an interplay between the Serre properties of the quotient, and those of its normalization and of the conductor subschemes. The precise statement is given in \autoref{proposition:Serre_property_of_quotient}.

\subsection{Acknowledgements}
Financial support was provided by the European Research Council under grant number 804334. 

\section{Preliminaries}

\subsection{Conventions and notations}
We will work with schemes that are not necessarily defined over a field, but we keep most of the usual terminology. 

A \emph{variety} is a separated reduced equidimensional excellent Noetherian scheme. Note that a variety in our sense might be reducible, even disconnected. A \emph{curve} (resp. a \emph{surface}, resp. a \emph{threefold}) is a variety of dimension one (resp. two, resp. three).

Let $X$ be a variety and $\sF$ a coherent $\sO_X$-module. We let $\sF^*:=\sHom_X(\sF,\sO_X)$ be the dual of $\sF$ and write $\sF^{[m]}:=\left(\sF^{\otimes m}\right)^{**}$. We say that $\sF$ is $S_i$ if $\depth_{\sO_{X,x}}\sF_x\geq \min\{i,\dim\sF_x\}$ for every $x\in X$.

If $X$ is a reduced Noetherian scheme, its \emph{normalization} is defined to be its relative normalization along the structural morphism $\bigsqcup_\eta\Spec(k(\eta))\to X$ where $\eta$ runs through the generic points of $X$. Recall that $X$ is normal if and only if it is regular in codimension one and $\sO_X$ is $S_2$.

If $X$ is a variety, we define $\bQ$-divisors and $\bQ$-Cartier divisors the usual way. We usually use these notions when $X$ is (demi-)normal.

If $R$ is a ring that is the quotient of a regular one, and if $X$ is a proper scheme over $\Spec R$, then $X$ has a dualizing complex $\omega_X^\bullet$ and a dualizing sheaf $\omega_X:=h^{-i}\left(\omega_X^\bullet\right)$ where $i:=\max\{j\mid h^{-j}\left(\omega_X^\bullet\right)\neq 0\}$. See \cite[\S 2.1]{Bhatt&Co_MMP_for_3folds_in_mixed_char} for more details. If $X$ is a variety and $\omega_X$ is invertible in codimension one, it defines a canonical divisor $K_X$.

A \textit{pair} is a data $(X,\Delta)$ where $X$ is a demi-normal variety, $\Delta$ a $\bQ$-divisor with coefficients in $[0,1]$ without component in common with $\Sing(X)$, such that $K_X+\Delta$ is $\bQ$-Cartier. If $f\colon Y\to X$ is a proper birational morphism from a normal variety, then there is a $\bQ$-divisor $\Delta_Y$ on $Y$ such that 
		$$K_Y+\Delta_Y=f^*(K_X+\Delta)$$
and $\Delta_Y$ is uniquely defined if we assume $f_*K_Y=K_X$, which we always do. Running through all such $f\colon Y\to X$, we define the singularity of the pair $(X,\Delta)$ (e.g. \textit{lc}, \textit{klt}) the usual way as in \cite{Kollar_Mori_Birational_geometry_of_algebraic_varieties}. See \cite[\S 2.5]{Bhatt&Co_MMP_for_3folds_in_mixed_char} for details.

Let $R$ be a discrete valuation ring (DVR) with maximal ideal $\pi R$. If $X$ is an $R$-scheme, we denote by $X_\pi$ the scheme-theoretic special fiber. We usually assume that $X$ is flat over $\Spec R$: if in addition $X$ is pure-dimensional, then every irreducible component of $X_\pi$ has dimension $\dim X-1$ by \cite[14.97]{Gortz_Wedhorn_AG_I}.

We will use many times the coarse classification of codimension two lc singularities \cite[Theorem 2.31]{Kollar_Singularities_of_the_minimal_model_program}: if $(X,\Delta)$ is an lc pair and $x\in X$ a codimension two point, then $\lfloor \Delta\rfloor$ is either regular or nodal at $x$. Notice that the proof given in \cite[Theorem 2.31]{Kollar_Singularities_of_the_minimal_model_program} holds in the generality of normal excellent Noetherian schemes.

Let $(X,D+\Delta)$ be a pair where $D$ is a reduced divisor. Then adjunction theory produces canonically a pair $(D^n,\Diff_{D^n}\Delta)$: see \cite[\S 4.1]{Kollar_Singularities_of_the_minimal_model_program}. In some cases adjunction already produces a different divisor on $D$, which will be important for us at a few places. See \cite[Definition 4.2]{Kollar_Singularities_of_the_minimal_model_program} for minimalistic assumptions necessary to perform adjunction.

\subsection{Quotients by finite equivalence relations}\label{section:quotients_by_finite_relations}
The theory of quotients by finite equivalence relations is developped in \cite{Kollar_Quotients_by_finite_equivalence_relations} and \cite[\S 9]{Kollar_Singularities_of_the_minimal_model_program}. For convenience, we recall the basic definitions and constructions that we will need.

Let $S$ be a base scheme, and $X,R$ two reduced $S$-schemes. An $S$-morphism $\sigma=(\sigma_1,\sigma_2)\colon R\to X\times_S X$ is a \textbf{set theoretic equivalence relation} (over $S$) if, for every geometric point $\Spec K\to S$, the induced map
		$$\sigma(\Spec K)\colon \Hom_S(\Spec K,R)\to \Hom_S(\Spec K,X)\times \Hom_S(\Spec K,X)$$
	is injective and an equivalence relation of the set $\Hom_S(\Spec K,X)$. We say in addition that $\sigma\colon R\to X\times_S X$ is \textbf{finite} if both $\sigma_i\colon R\to X$ are finite morphisms.
	
	Suppose that $\sigma\colon R\hookrightarrow X\times_S X$ is a reduced closed subscheme. Then there is a minimal set theoretic equivalence relation generated by $R$: see \cite[9.3]{Kollar_Singularities_of_the_minimal_model_program}. Even if both $\sigma_i\colon R\to X$ are finite morphisms, the resulting pro-finite relation may not be finite: achieving transitivity can create infinite equivalence classes. 
	
	Let $f\colon X'\to X$ be a finite $S$-morphism, and $\tau$ an involution of $X'$. We denote by $R_X(\tau)\rightrightarrows X$ the set theoretic equivalence relation induced by the closed subscheme $(f\times f)\left(\Delta_{X'/S}\cup \Gamma_\tau\right)$ of $X\times_SX$, where $\Delta_{X'/S}\subset X'\times_S X'$ is the diagonal and $\Gamma_\tau\subset X'\times_SX'$ is the graph of $\tau$. If there is no  risk of confusion, we drop the subscript and simply write $R(\tau)$.

	Let $(\sigma_1,\sigma_2)\colon R\to X\times_S X$ be a finite set theoretic equivalence relation. A \textbf{geometric quotient} of this relation is an $S$-morphism $q\colon X\to Y$ such that
		\begin{enumerate}
			\item $q\circ \sigma_1=q\circ \sigma_2$,
			\item $(Y,q\colon X\to Y)$ is initial in the category of algebraic spaces for the property above, and
			\item $q$ is finite.
		\end{enumerate}
We note that if the quotient $Y$ exists, then it is local for the flat topology of $Y$ \cite[Corollary 9.11]{Kollar_Singularities_of_the_minimal_model_program}.		
		
The most important result for us is the following theorem of Witaszek \cite[Theorem 1.4]{Witaszek_Keel_theorem_and_quotients_in_mixed_char}. Assume $X$ is separated of finite type over an excellent scheme $S$ (in our applications, $S$ will be the spectrum of a DVR), and $R\to X\times_S X$ a finite set-theoretic equivalence relation over $S$. If the geometric quotient $X_\bQ/R_\bQ$ exists, then the geometric quotient $X/R$ also exists (as an algebraic space) and is of finite type over $S$. If $X$ is projective over $S$, then $X$ has the Chevalley--Kleiman property, so $X/R$ is a scheme by \cite[Corollary 48]{Kollar_Quotients_by_finite_equivalence_relations}.

\begin{lemma}\label{lemma:quotient_of_S_2_is_S_2}
Let $X$ be a reduced equidimensional scheme that is separated of finite type over a field of positive characteristic, and $D\subset X$ a reduced scheme of pure codimension one. Let $\tau$ be an involution of the normalization $D^n$, such that $R(\tau)\rightrightarrows X$ is a finite set theoretic equivalence relation, and let $p\colon X\to Y:=X/R(\tau)$ be the quotient.

If $X$ is $S_2$, then so is $Y$.
\end{lemma}
\begin{proof}
The argument is contained in the proof of \cite[3.4.1]{Posva_Gluing_for_surfaces_and_threefolds}, although it is not phrased in our generality. Since $X$ is reduced, the morphism $p\colon X\to Y$ factors through $Y_\text{red}$. It follows by the universal property of $p$ that $Y_\text{red}\to Y$ is an isomorphism, so $Y$ is reduced and in particular $S_1$. Thus by \cite[1.8]{Hartshorne_Generalized_divisors_and_biliaison}, it is sufficient to show that for every $U\subset Y$ open and $Z\subset U$ closed of codimension $\geq 2$, the restriction map $\sO_Y(U)\to \sO_Y(U\setminus Z)$ is surjective. The quotient is Zariski-local, so we may assume that $U=Y$.

Take $s\in \sO(U\setminus Z)$. Since $X$ is $S_2$ and $p^{-1}(Z)$ has codimension $\geq 2$, the section $p^*s$ extends to a global section over $X$. By construction $p^*s|_{D^n}$ is $\tau$-invariant on an open dense subset, hence it is globally $\tau$-invariant. Thus $p^*s$ descends to a global section of $Y$ \cite[Lemma 9.10]{Kollar_Singularities_of_the_minimal_model_program} which is an extension of $s$.
\end{proof}

\subsection{Demi-normal varieties and slc pairs}\label{section:demi_normal_and_slc}
We give the definitions and the basic properties of demi-normal varieties and of slc pairs over an excellent DVR.

\begin{definition}
A one-dimensional Noetherian local ring $(R,\fm)$ is called a \textbf{node} if there exists a ring isomorphism $R\cong S/(f)$, where $(S,\fn)$ is a regular two-dimensional local ring and $f\in \fn^2$ is an element that is not a square in $\fn^2/\fn^3$.

A locally Noetherian reduced scheme (or ring) is called \textbf{nodal} if its codimension one local rings are regular or nodal. It is called \textbf{demi-normal} if it is $S_2$ and nodal.
\end{definition}

Let $X$ be a reduced scheme with normalization $\pi\colon \bar{X}\to X$. The conductor ideal of the normalization is defined as
		$$\fI:=\sHom_X(\pi_*\sO_{\bar{X}},\sO_X).$$
It is an ideal in both $\sO_X$ and $\sO_{\bar{X}}$. We let
		$$D:=\Spec_X\sO_X/\fI,\quad \bar{D}:=\Spec_{\bar{X}}\sO_{\bar{X}}/\fI,$$
and call them the \textit{conductor subschemes}. The basic properties of $D$ and $\bar{D}$, in case $X$ is demi-normal, are given by the following lemma.

\begin{lemma}\label{lemma:properties_of_nodes}
Notations as above. Assume that $X$ is excellent and demi-normal. Then:
	\begin{enumerate}
		\item $D$ and $\bar{D}$ are reduced of pure codimension $1$; in particular $D$ is the closure of the singular codimension one points of $X$.
		\item If $\eta\in D$ is a generic point such that $\Char k(\eta)\neq 2$, the morphism $\bar{D}\to D$ is \'{e}tale of degree $2$ in a neighborhood of $\eta$.
	\end{enumerate}
Assume in addition that $X$ has a dualizing sheaf $\omega_X$. Then:
	\begin{enumerate}\setcounter{enumi}{2}
		\item $\omega_X$ is invertible in codimension one, so $X$ has a canonical divisor $K_X$;
		\item Let $\Delta$ be a $\bQ$-divisor on $X$ with no component supported on $D$, and such that $K_X+\Delta$ is $\bQ$-Cartier. If $\bar{\Delta}$ denotes the divisorial part of $\pi^{-1}(\Delta)$, then there is a canonical isomorphism
				$$\pi^*\omega_X^{[m]}(m\Delta)\cong \omega_{\bar{X}}^{[m]}(m\bar{D}+m\bar{\Delta})$$
		for $m$ divisible enough.
	\end{enumerate}
\end{lemma}
\begin{proof}
This is proved as in \cite[5.2, 5.7]{Kollar_Singularities_of_the_minimal_model_program} (notice that the adjunction property is proved in \cite[Definition 4.2, 5.7]{Kollar_Singularities_of_the_minimal_model_program} in the generality we need).
\end{proof}

\begin{lemma}\label{lemma:involution_from_normalization}
Notations as above. Assume that $\Char k(\eta)\neq 2$ for every generic point $\eta$ of $D$. Then:
	\begin{enumerate}
		\item the induced morphism of normalizations $\bar{D}^n\to D^n$ is the geometric quotient by a Galois involution $\tau$;
		\item if $K_X+\Delta$ is $\bQ$-Cartier, then $\tau$ is a log involution of $(\bar{D}^n,\Diff_{\bar{D}^n}\bar{\Delta})$.
	\end{enumerate}
\end{lemma}
\begin{proof}
See \cite[2.2.5]{Posva_Gluing_for_surfaces_and_threefolds}.
\end{proof}

\begin{definition}
We say that $(X,\Delta)$ is a \textbf{semi-log canonical (slc) pair} over a DVR if: $X$ is demi-normal of finite type over a DVR, $\Delta$ is a $\bQ$-divisor with no components along $D$, $K_X+\Delta$ is $\bQ$-Cartier, and the normalization $(\bar{X}, \bar{D}+\bar{\Delta})$ is an lc pair.
\end{definition}

\subsubsection{Semi-normality, weak normality and demi-normality}\label{section:weak_normality}
In this subsection, we compare the notions of semi-normality, weak normality and demi-normality.

For the definitions of \textbf{seminormal} and \textbf{weakly normal} factorizations, we refer to \cite[Appendix I.7.2]{Kollar_Rational_curves}. For us, the relevant property is the following: given a finite surjective morphism of reduced schemes $Z\to X$, we that $X$ is seminormal in $Z$ (resp. weakly normal in $Z$) if, given a factorization $Z\to X'\to X$ where $X'\to X$ is an homeomorphism inducing isomorphisms on residues fields (resp. an homeomorphism inducing purely inseparable extension of residue fields), it holds that $X'=X$. These properties are local on $X$.

Given a reduced excellent scheme $X$, we say that $X$ is seminormal (resp. weakly normal) if the normalization $X^n\to X$ is semi-normal (resp. weakly normal).

\begin{lemma}\label{lemma:pushout_and_weak_normality}
Consider a pushout diagram of reduced $\bF_p$-schemes 
		$$\begin{tikzcd}
		X_1\arrow[r, hook]\arrow[d, "q"] & X_2\arrow[d, "p"] \\
		X_3\arrow[r, hook] & Y
		\end{tikzcd}$$
where the horizontal arrows are closed embeddings, $p$ is finite surjective and an isomorphism over every generic point of $X$, and $q$ is finite surjective. Then $Y$ is weakly normal in $X_2$ if and only $X_3$ is weakly normal in $X_1$.
\end{lemma}
\begin{proof}
As all morphisms are affine, we may assume that all schemes are affine. We use the following characterization of weak normality for $\bF_p$-schemes: a finite surjective morphism $Z\to Z'$ is weakly normal if and only if $\sO_{Z'}$ is $p$-closed in $\sO_Z$ \cite[Corollary to Theorem 1]{Yanagihara_Weakly_normal_ring_extensions}.

Assume that $Y$ is weakly normal in $X_2$. Pick an element $s\in \sO_{X_1}$ such that $s^p\in \sO_{X_3}$. Choose any lift $u\in \sO_{X_2}$ of $s$. Then the pair $(u^p,s^p)\in \sO_{X_2}\times \sO_{X_3}$ glues to an element of $\sO_Y$ that is a $p$-th power of $u$. Therefore $u\in \sO_Y$, and its restriction to $X_3$ is precisely $s$.

Conversely, assume that $X_3$ is weakly normal in $X_1$. Pick an element $v\in \sO_{X_2}$ such that $v^p\in \sO_Y$. If $t\in \sO_{X_1}$ is the restriction of $v$ to $X_1$, then $t^p\in \sO_{X_3}$. Thus $t\in \sO_{X_3}$ already. Therefore the pair $(v,t)\in \sO_{X_2}\times\sO_{X_3}$ glues to $v\in \sO_Y$.
\end{proof}

Next let us recall the structure theorems for seminormal and weakly normal ring extensions, as given in \cite{Traverso_Seminormality_and_Picard_group} and \cite{Yanagihara_Weakly_normal_ring_extensions}. Let $A\subset B$ be a finite extension of rings, $\fp$ be a prime ideal of $A$ and $\fp_1,\dots,\fp_n$ the prime ideals of $B$ lying over $\fp$. We have a natural map $k(\fp)\to \prod_ik(\fp_i)$. Consider the pullback diagram of rings
		$$\begin{tikzcd}
		A^+\arrow[r] \arrow[d] & B\arrow[d] \\
		k(\fp)\arrow[r] & \prod_{i=1}^nk(\fp_i).
		\end{tikzcd}$$
We call $A^+$ the \textbf{gluing of $B$ over $\fp\subset A$}. It is an intermediate extension of $A\subset B$, with a unique prime ideal $\fp^+$ over $\fp\subset A$ whose residue field is $k(\fp^+)=k(\fp)$. 

Furthermore, let $k'$ be the largest subfield of $\prod_ik(\fp_i)$ which is a purely inseparable extension of $k(\fp)$. Consider the pullback diagram
		$$\begin{tikzcd}
		A^*\arrow[r] \arrow[d] & B\arrow[d] \\
		k'\arrow[r] & \prod_{i=1}^nk(\fp_i).
		\end{tikzcd}$$
We call $A^*$ the \textbf{weak gluing of $B$ over $\fp\subset A$}.  It is an intermediate extension of $A\subset B$, with a unique prime ideal $\fp^*$ over $\fp\subset A$ whose residue field is $k(\fp^*)=k'$. 

\begin{theorem}[{\cite[Theorem 2.1]{Traverso_Seminormality_and_Picard_group} and \cite[Theorem 3]{Yanagihara_Weakly_normal_ring_extensions}}]\label{thm:structure_of_seminormal_ext}
Let $A\subset B$ be a finite extension of Noetherian rings. Then $A$ is seminormal (resp. weakly normal) in $B$ if and only if then there exists a finite sequence of finite extensions
		$$A=A_0\subset A_1\subset \dots \subset A_n=B$$
where $A_{i}$ is the gluing (resp. the weak gluing) of $A_{i+1}$ over a prime ideal of $A_i$.
\end{theorem}

\begin{lemma}\label{lemma:nodes_are_weakly_normal}
Let $(R,\fm)$ be an excellent nodal local ring. Then $R$ is seminormal in its normalization $R^\nu$, and $R$ is the gluing of $R^\nu$ over $\fm\subset R$. Furthermore, $R$ is weakly normal in $R^\nu$ if and only if $R$ is a separable node (see \cite[Definition 3.2.4]{Posva_Gluing_for_surfaces_and_threefolds}).
\end{lemma}
\begin{proof}
By \autoref{lemma:properties_of_nodes} the conductor of $R\subset R^\nu$ is a radical ideal of $R^\nu$. Thus by \cite[Appendix I.7.2.5]{Kollar_Rational_curves} we obtain that $R$ is seminormal in $R^\nu$ if and only if $k(\fm)$ is seminormal in $R^\nu/\fm R^\nu$. This is true as the latter ring is a two-dimensional $k(\fm)$-vector space \cite[\S 3.2]{Posva_Gluing_for_surfaces_and_threefolds}.

Let $R^+$ be the gluing of $R^\nu$ over $\fm\subset R$. By construction $R\subset R^+$ induces an homeomorphism of spectra and isomorphisms on residue fields. Thus $R=R^+$ by seminormality.

In particular we have a pullback diagram
		\begin{equation}\label{eqn:gluing_diagram_for_nodes}
		\begin{tikzcd}
		R\arrow[r]\arrow[d] & R^\nu \arrow[d]\\
		k(\fm) \arrow[r] & R^\nu/\fm R^\nu.
		\end{tikzcd}
		\end{equation}
By definition $R/\fm R$ is a purely inseparable field extension of $k(\fm)$ if and only if $R$ is an inseparable node. Thus $R$ is the weak gluing of $R^\nu$ above $\fm\subset R$ if and only if $R$ is a separable node.
\end{proof}

\begin{corollary}
Let $X$ be an excellent demi-normal scheme with normalization $\pi\colon \bar{X}\to X$. Then $X$ is seminormal (in $\bar{X}$). Moreover $X$ is weakly normal (in $\bar{X}$) if and only if $X$ has only separable nodes.
\end{corollary}
\begin{proof}
The statements are true in codimension one by \autoref{lemma:nodes_are_weakly_normal}. The $S_2$ property of $X$ implies that they hold globally, see \cite[\S 1, Theorem 1]{Yanagihara_Weakly_normal_ring_extensions}.
\end{proof}

\begin{corollary}\label{corollary:descend_section_by_gluing}
Let $X$ be an excellent demi-normal scheme with normalization $\pi\colon \bar{X}\to X$. A section $s\in \pi_*\sO_{\bar{X}}$ belongs to $\sO_X$ if and only if, for every codimension one point $\eta\in X$, we have
				$$s|_{\bar{X}_\eta}\in \im \left[ k(\eta)\to H^0(\bar{X}_\eta,\sO_{\bar{X}_\eta})\right].$$
\end{corollary}
\begin{proof}
Let $s\in \pi_*\sO_{\bar{X}}$ be a section satisfying the restriction condition of the statement. Looking at the pullback diagram \autoref{eqn:gluing_diagram_for_nodes}, with $R=\sO_{X,\eta}$ for an arbitrary codimension one point $\eta$, we deduce that $s\in \sO_{X,\eta}$. Since $X$ is $S_2$ we obtain that $s\in \sO_X$. The converse is clear.
\end{proof}

\subsubsection{Slc compactification for surfaces}
It is an interesting and difficult question in general whether a quasi-projective pair with slc singularities can be embedded into a projective one. It is easy to embed a quasi-projective demi-normal scheme into a projective one (see \cite[3.7.1]{Posva_Gluing_for_surfaces_and_threefolds}), so the difficulty comes from the $\bQ$-Cartier condition on the log canonical divisor. For lc pairs in characteristic zero, this is achieved in \cite[Corollary 1.2]{Hacon_Xu_Existence_of_lc_closures}. We shall need the simple case of slc surfaces over an excellent DVR: we present it now.

\begin{notation}
In what follows, we let $T$ be the spectrum of an excellent DVR, with closed point $\pi$ and residue characteristic different from $2$. (The case of residue characteristic $2$ can be treated in the same way, taking in account the inseparable nodes \cite[\S 3.2]{Posva_Gluing_for_surfaces_and_threefolds} that may appear. We exclude it to avoid additional complications, since we will stay away from characteristic $2$ for most of this article.)
\end{notation}

\begin{lemma}\label{lemma:quotient_surfaces}
Let $S$ be a surface that is quasi-projective over $T$, and $R\hookrightarrow S\times_T S$ a finite equivalence relation. Then the geometric quotient $S/R$ exists (as a scheme), and is a surface that is of finite type over $T$. If $S$ is proper over $T$, then so is $S/R$.
\end{lemma}
\begin{proof}
If $T$ is of equicharacteristic, the fact that the geometric quotient $S\to S/R$ exists (as an algebraic space) follows either from \cite[Theorem 6]{Kollar_Quotients_by_finite_equivalence_relations} or \cite[Proposition 33]{Kollar_Quotients_by_finite_equivalence_relations}. If $T$ is of mixed characteristic, this follows from \cite[Theorem 1.4]{Witaszek_Keel_theorem_and_quotients_in_mixed_char} together with the existence of the quotient for the generic fiber, which follows from \cite[Proposition 33]{Kollar_Quotients_by_finite_equivalence_relations}. The fact that the algebraic space $S/R$ is actually a scheme follows from \cite[9.31]{Kollar_Singularities_of_the_minimal_model_program}.

By the universal property of quotients, the morphism $S\to T$ factors through $S/R$. Since $S$ is of finite type over $T$ and finite over $S/R$, we obtain that $S/R$ is of finite type over $T$ \cite{Artin_Tate_Note_on_finite_extensions}. Properness descends the quotient by \cite[09MQ, 03GN]{StacksProject}.
\end{proof}

\begin{remark}
In general, even if $S$ is projective over $T$, we cannot expect the same property for $S/R$: see \cite[Example 14]{Kollar_Quotients_by_finite_equivalence_relations}.
\end{remark}

\begin{lemma}\label{lemma:quotient_surface_mixed_char}
Let $S$ be a surface that is flat and quasi-projective over $T$. Let $D+\Delta$ be a boundary on $S$, where $D$ has coefficients one, such that $(S,D+\Delta)$ is lc. Let $\tau$ be a log involution of the pair $(\bar{D}^n,\Diff_{D^n}(\Delta))$ over $T$. Then:
	\begin{enumerate}
		\item The equivalence relation $R(\tau)\rightrightarrows S$ is finite.
		\item If $q\colon S\to S/R(\tau)=S_q$ is the quotient morphism, then $(S_q,\Delta_q=q_*\Delta)$ is an slc surface of finite type over $T$.
		\item $(S,D+\Delta,\tau)$ is the normalization of $(S_q,\Delta_q)$ in the sense of \autoref{section:demi_normal_and_slc}.
	\end{enumerate}
\end{lemma}
\begin{proof}
The fact that $R(\tau)\rightrightarrows S$ is finite can be proved as in \cite[Proof of Theorem 4.1.1]{Posva_Gluing_for_surfaces_and_threefolds}. In that reference it is assumed that $S$ is defined over a field, but this is not necessary to obtain the finiteness of $R(\tau)$. It follows from \autoref{lemma:quotient_surfaces} that the geometric quotient $q\colon S\to S/R(\tau)=S_q$ exists.

The fact that $S_q$ is demi-normal with normalization $(S,D,\tau)$ follows from \cite[3.4.1]{Posva_Gluing_for_surfaces_and_threefolds}. Once again, it is assumed there that $S$ is defined over a field: this is not necessary for the proof if we already know that the quotient exists. The fact that $(S_q,\Delta_q)$ is slc follows from \cite[3.2.7]{Posva_Gluing_for_surfaces_and_threefolds} and the fact that $(S,D+\Delta)$ is slc.
\end{proof}

\begin{proposition}\label{prop:compactification_slc_surface}
Let $S_0\to T$ be a flat quasi-projective morphism from a demi-normal surface, and assume $\Delta_0$ is a boundary such that $(S_0,\Delta_0+S_{0,\pi})$ is slc.

Then there exist a proper morphism $S\to T$ from a demi-normal surface $S$, a boundary $\Delta$ such that $(S,\Delta+S_\pi)$ is slc, and an open immersion $(S_0,\Delta_0)\hookrightarrow (S,\Delta)$ over $T$.
\end{proposition}
\begin{proof}
Following the proof of \cite[Theorem 3.7.1]{Posva_Gluing_for_surfaces_and_threefolds} and using \autoref{lemma:quotient_surfaces} at the appropriate places, we can embed $S_0$ inside a demi-normal surface $S'$ that is projective over $T$, and assume that the singular codimension one points of $S'$ are contained in $S_0$. 

We take the normalization of $S'$ and obtain a normal surface $\bar{S}$ that is proper over $T$, and that contains the normalization $S_0^n$ as an open subset. If $\bar{D}\subset \bar{S}$ and $D_0\subset S_0^n$ are the conductor divisors, then $\bar{D}$ is the closure of $D_0$. Let $\bar{\Delta}$ be the closure in $\bar{S}$ of the strict transform of $\Delta_0$ in $S_0^n$. By blowing-up repeatedly on the complement of $S_0^n$, we may assume that $(\bar{S},\bar{\Delta}+\bar{S}_\pi)$ is log canonical in a neighbourhood of $\bar{S}\setminus S_0^n$ (this is achievable by \cite[Theorem 2.25]{Kollar_Singularities_of_the_minimal_model_program}).

By construction, the normal proper log curve $(\bar{D}^n,\Diff_{\bar{D}^n}(\bar{\Delta}+\bar{S}_\pi))$ is endowed with a birational involution $\tau_0\colon \bar{D}^n\dashrightarrow \bar{D}^n$, that (co)-restricts to a solid involution of the dense open set $U:=\bar{D}^n\times_{\bar{S}}S_0^n$. Moreover $\tau_0$ preserves the divisor $\Diff_{\bar{D}^n}(\bar{\Delta}+\bar{S}_\pi)\cap U$. Since $\bar{D}^n$ is a normal proper curve, the birational self-map $\tau_0$ extends uniquely to an involution $\tau$ of $\bar{D}^n$. We may blow-up $\bar{S}$ even more, so that $(\bar{\Delta}+ \bar{S}_\pi)\cap \bar{D}$ is contained in $S_0^n$. If we do so, then $\tau$ becomes an involution of the log pair $(\bar{D}^n,\Diff_{\bar{D}^n}(\bar{\Delta}+\bar{S}_\pi))$.

By \autoref{lemma:quotient_surface_mixed_char}, we may glue the lc pair $(\bar{S},\bar{D}+\bar{\Delta}+\bar{S}_\pi)$ along the involution $\tau$, and obtain an slc surface pair $(S,\Delta+S_\pi)$ that is proper over $T$. Quotients are Zariski-local \cite[Corollary 9.11]{Kollar_Singularities_of_the_minimal_model_program}, so the pair $(S_0,\Delta_0+S_{0,\pi})$ admits an open embedding into $(S,\Delta+S_\pi)$.
\end{proof}

\section{Gluing for families of surfaces}

The goal of this section is the proof of \autoref{theorem:gluing_for_surface_families}. We consider the following situation (where, compared to \autoref{theorem:gluing_for_surface_families}, we modify slightly our notations):
\begin{notation}\label{notation:slc_family}
We let:
	\begin{itemize}
		\item $R$ be an excellent DVR of mixed characteristic $(0,p>0)$;
		\item $X\to \Spec R$ be a flat proper morphism with connected fibers from an equidimensional normal scheme $X$ of any dimension;
		\item $D$ be an effective reduced Weil divisor on $X$ and $\Delta$ be an effective $\mathbb{Q}$-Weil divisor on $X$, and we assume that $(X,D+\Delta+X_\pi)$ is lc. In particular, $X_\pi$ is reduced and has no component contained in the support of $D+\Delta$.
		\item We denote by $\bar{D}\to D$ the normalization of $D$, by $\bar{D}_\pi$ the special fiber of $\bar{D}\to\Spec R$ and by $\bar{D}_\pi^n$ its normalization.
		\item Finally, we assume that there exists an involution $\tau$ of the log pair $(\bar{D}, \Diff_{\bar{D}}\Delta)$ over $\Spec R$. 
	\end{itemize}
\end{notation}

Our goal is to prove that $X/R(\tau)$ exists as a scheme and has an slc structure.

\subsection{Existence of the quotient}\label{section:Existence_of_quotient}

To begin with, we prove that $X/R(\tau)$ exists as a scheme.

\begin{lemma}\label{lemma:involution_respects_fibration}
Let $Z\to \Spec R$ be a morphism and $\tau$ an involution of $Z$ over $\Spec R$. Then $\tau$ restricts to involutions of $Z_\pi$, $\red(Z_\pi)$ and $Z\otimes_{R} \Frac(R)$.
\end{lemma}
\begin{proof}
The ideal of $Z_\pi$ is $\pi \sO_Z$, so by looking at the exact sequence 
		$$0\to \pi\sO_Z\to \sO_Z\to \sO_{Z_\pi}\to 0$$
we see that $\tau$ restricts to an involution of $Z_\pi$ if and only if $\tau(\pi)\subseteq \pi\sO_Z$. But $\tau(\pi)=\pi$, so this is immediate. So $\tau$ descends to $Z_\pi$, and thus to $\red(Z_\pi)$. Since $\tau(\pi)=\pi$, it also holds that $\tau$ descends to an involution of $\sO_Z\otimes_R R[1/\pi]=\sO_{Z\otimes_{ R} \Frac(R)}$.
\end{proof}

\begin{lemma}\label{lemma:lc_center_on_families}
Let $(X,D+\Delta)\to \Spec R$ be as in \autoref{notation:slc_family}. If $E$ is a divisor over $X$ whose center $c_X(E)$ belongs to the special fiber $X_\pi$, then $a(E;X,D+\Delta)\geq 0$.
\end{lemma}
\begin{proof}
The proof is the same as \cite[2.14]{Kollar_Families_of_varieties_of_general_type}. Suppose that $E$ appears as a divisor on a proper birational model $\pi\colon Y\to X$. Write $b_E:=\coeff_E\pi^*X_\pi$. Since $\pi^*X_\pi$ is Cartier and effective, $b_E$ is a non-negative integer. If $c_X(E)\subset X_\pi$ then $b_E$ is actually a positive integer. Then:
		$$-1\leq a(E;X,D+\Delta+X_\pi)=a(E;X,D+\Delta)-b_E$$
so $a(E;X,D+\Delta)\geq 0$.
\end{proof}

\begin{lemma}\label{lemma:intersection_has_expected_dim}
In the situation of \autoref{notation:slc_family}, every irreducible component of $D$ dominates $\Spec R$, and the irreducible components of $D\cap X_\pi$ have dimension $\dim D-1$.
\end{lemma}
\begin{proof}
Since $X$ is flat over $\Spec R$, every component of $X_\pi$ is a divisor by \cite[14.97]{Gortz_Wedhorn_AG_I}. Since $(X,D+\Delta+X_\pi)$ is lc, no component of the boundary has coefficient $>1$. Thus $D$ does not contain any component of $X_\pi$, and so every component of $D$ dominates $\Spec R$ and so $D\to \Spec R$ is flat \cite[III.9.7]{Hartshorne_Algebraic_Geometry}. Applying \cite[14.97]{Gortz_Wedhorn_AG_I} again yields the result.
\end{proof}

\begin{lemma}\label{lemma:special_fiber_is_reduced}
The special fiber $\bar{D}_\pi$ of $\bar{D}\to \Spec R$ is reduced.
\end{lemma}
\begin{proof}
Since $\bar{D}$ is normal and $\bar{D}_\pi$ is an hypersurface, $\bar{D}_\pi$ is $S_1$. Thus we only need to show that $\bar{D}_\pi$ is generically reduced.

The generic points of $\bar{D}_\pi$ dominates the generic points of the intersection $D\cap X_\pi$, and by \autoref{lemma:intersection_has_expected_dim} these points have codimension two in $X$. Since $(X,D+\Delta+X_\pi)$ is lc, the classification of codimension two lc singularities \cite[Corollary 2.32]{Kollar_Singularities_of_the_minimal_model_program} shows that around the generic points of $D\cap X_\pi$, the divisors $D$ and $X_\pi$ are regular and meet transversally. Hence $\bar{D}\to D$ is an isomorphism above those generic points, with $\bar{D}_\pi$ isomorphic to the regular $D\cap X_\pi$. This shows that $\bar{D}_\pi$ is generically reduced.
\end{proof}

\begin{lemma}\label{lemma:extension_of_involution}
The involution $\tau$ of $(\bar{D}, \Diff_{\bar{D}}\Delta)$ induces an involution $\sigma$ of the lc pair $(\bar{D}_\pi^n,\Gamma)$, where $\Gamma$ is defined by the adjunction formula $(K_{\bar{D}}+\Diff_{\bar{D}}(\Delta)+\bar{D}_\pi)|_{\bar{D}_\pi^n}=K_{\bar{D}_\pi^n}+\Gamma$.

Moreover, the equivalence relation $R_{X_\pi}(\sigma)\rightrightarrows X_\pi$ is equal to the restriction of the equivalence relation $R_X(\tau)\rightrightarrows X$ to $X_\pi$.
\end{lemma}
\begin{proof}
By \autoref{lemma:involution_respects_fibration} and \autoref{lemma:special_fiber_is_reduced} the involution $\tau$ descends to an involution $\sigma'$ on $\bar{D}_\pi$. By the universal property of normalization, we obtain an involution $\sigma$ of the normalization $\bar{D}_\pi^n$ that makes the diagram
		\begin{equation}\label{eqn:lift_of_involution_to_normalization}
		\begin{tikzcd}
		\bar{D}^n_\pi \arrow[d]\arrow[r, "\sigma"] & \bar{D}^n_\pi\arrow[d] \\
		\bar{D}_\pi\arrow[r, "\sigma'"] & \bar{D}_\pi
		\end{tikzcd}
		\end{equation}
	commutative.

Since the $\bQ$-Cartier divisors $K_{\bar{D}}+\Diff_{\bar{D}}\Delta$ and $\bar{D}_\pi$ are $\tau$-invariant, so is their sum. Hence the pullback of $K_{\bar{D}}+\Diff_{\bar{D}}+\bar{D}^n_\pi$ to $\bar{D}_\pi^n$ is $\sigma$-invariant. The pair $(\bar{D}_\pi^n,\Gamma)$ is lc by adjunction \cite[Lemma 4.8]{Kollar_Singularities_of_the_minimal_model_program}. It is clear from the construction that $R(\sigma')\rightrightarrows X_\pi$ is equal to the restriction of $R_X(\tau)$ to $X_\pi$, so we only need to compare $R_{X_\pi}(\sigma)$ and $R_{X_\pi}(\sigma')$. Since \autoref{eqn:lift_of_involution_to_normalization} is commutative, we see that the two equivalence relations are the same.
\end{proof}

\begin{lemma}\label{lemma:adjunction_route_along_special_fiber}
The special fiber $X_\pi$ is reduced, regular at the generic points of $D\cap X_\pi$ and at worst nodal at other codimension one points. Moreover if $X_\pi^n\to X_\pi$ is the normalization morphism then:
	\begin{enumerate}
		\item it is an isomorphism over the generic points of $D\cap X_\pi$,
		\item the pair $(X_\pi^n,\Diff_{X_\pi^n}(D+\Delta))$ is lc, and
		\item the strict transform of $D\cap X_\pi$ is contained in $\lfloor \Diff_{X_\pi^n}(D+\Delta)\rfloor$.
	\end{enumerate}
\end{lemma}
\begin{proof}
The hypersurface $X_\pi$ of $X$ is $S_1$. It is regular at the generic points of $D\cap X_\pi$ and at worst nodal at other codimension one points by \cite[Corollary 2.32]{Kollar_Singularities_of_the_minimal_model_program}. To decide whether the strict transform of $D\cap X_\pi$ appears in $\lfloor \Diff_{X_\pi^n}(D+\Delta)\rfloor$ is a local question around the generic points of $D\cap X_\pi$. So the first and third points hold. The pair $(X_\pi^n,\Diff_{X_\pi^n}(D+\Delta))$ is lc by adjunction.
\end{proof}

\begin{lemma}\label{lemma:two_routes_give_the_same_pair}
Denote by $E\subset \lfloor \Diff_{X_\pi^n}(D+\Delta)\rfloor$ the strict transform of $D\cap X_\pi$, and by $(E^n,\Theta)$ the lc pair obtained from $(X_\pi^n,\Diff_{X_\pi^n}(D+\Delta))$ by adjunction. Then there exists a log isomorphism $f\colon (E^n,\Theta)\cong (\bar{D}_\pi^n,\Gamma)$ such that the diagram
		\begin{equation}\label{eqn:crepant_diagram}
		\begin{tikzcd}
		E^n \arrow[rr, "f"]\arrow[dr] && \bar{D}^n_\pi \arrow[dl] \\
		& X &
		\end{tikzcd}
		\end{equation}
commutes.
\end{lemma}
\begin{proof}
As observed in \autoref{lemma:special_fiber_is_reduced} and \autoref{lemma:adjunction_route_along_special_fiber}, the two morphisms 
		$$E\to D\cap X_\pi \leftarrow \bar{D}_\pi$$
are birational and finite. This induces finite birational morphisms
		$$E^n\to (D\cap X_\pi)^n\leftarrow \bar{D}_\pi^n$$
of normal schemes, so these must be isomorphisms. Thus we obtain an isomorphism $f\colon E^n\cong \bar{D}_\pi^n$ that commutes with the morphisms to $X$. Since the divisors $\Theta$ and $\Gamma$ are defined by adjunction, we see that $f$ is a log isomorphism.
\end{proof}

\begin{proposition}\label{proposition:gluing_exists}
In the situation of \autoref{notation:slc_family}, assume that $X$ is proper over $\Spec R$ and that $K_X+D+\Delta$ is ample over $\Spec R$. Assume also that
	\begin{enumerate}
				\item $\dim X=3$, or
				\item $\dim X=4$, and $X_\pi$ is $S_2$ in a neighborhood of $D\cap X_\pi$ and the residue field of $R$ is perfect of characteristic $>5$.
	\end{enumerate}
Then the quotient $X/R(\tau)$ exists as a demi-normal scheme that is flat and proper over $\Spec R$.
\end{proposition}
\begin{proof}
As explained in \autoref{section:quotients_by_finite_relations}, it follows from \cite{Witaszek_Keel_theorem_and_quotients_in_mixed_char} that the geometric quotient $X/R(\tau)$ exists as a scheme as soon as $R(\tau)$ is finite and the geometric quotient $X_\bQ/R(\tau)_\bQ$ exists (notice that $X$ is projective over $\Spec R$ by assumption, so $X/R(\tau)$ is indeed a scheme).

\autoref{lemma:involution_respects_fibration} implies that $\tau$ induces an involution $\tau_\bQ$ on the generic fiber $(\bar{D}_\bQ,\Diff_{\bar{D}_\bQ}(\Delta_\bQ))$ and $R(\tau_\bQ)=R(\tau)|_{X_\bQ}$.

Since $K_X+D+\Delta$ is ample over $\Spec R$, we see that $(X_\bQ,D_\bQ+\Delta_\bQ)$ is a projective lc pair over a field of characteristic zero with ample log canonical divisor a log involution $\tau_\bQ$ on $(\bar{D}_\bQ,\Diff_{\bar{D}_\bQ}(\Delta_\bQ))$. By \cite[Theorem 5.13]{Kollar_Singularities_of_the_minimal_model_program} the quotient $X_\bQ/R(\tau)_\bQ$ exists.

We still have to show that $R(\tau)$ is finite. Since it respects the fibration to $\Spec R$, we only need to show that $R(\tau)|_{X_\pi}$ is finite. By \autoref{lemma:extension_of_involution} it is equivalent to show that $R_{X_\pi}(\sigma)\rightrightarrows X_\pi$ is finite. By \autoref{lemma:two_routes_give_the_same_pair} and the commutativity of \autoref{eqn:crepant_diagram}, we may transport $\sigma$ to an involution of $(E^n,\Theta)$. We are now in situation to apply \cite[4.1.1, 4.1.3, 5.3.2]{Posva_Gluing_for_surfaces_and_threefolds} which shows that $R_{X_\pi^n}(\sigma)$ is finite.

To go from $X_\pi^n$ to $X_\pi$, observe that the equivalence generated by $\sigma$ on $X_\pi^n$, respectively on $X_\pi$, is trivial away from the support of $E$, respectively away from the support of $D\cap X_\pi$. Now $X_\pi$ is $R_1$ is a neigborhood of $D\cap X_\pi$ by \autoref{lemma:two_routes_give_the_same_pair}. If it is also $S_2$  then $X_\pi^n\to X_\pi$ is an isomorphism in a neighborhood of $E$, and we deduce that $R_{X_\pi}(\sigma)$ is finite.

If $\dim X=3$, then we do not need the fact that $X_\pi$ is $S_2$. In this case $X_\pi$ is a reduced surface that is regular in codimension one and such that $\omega_{X_\pi}^{[m]}(m(D+\Delta)|_{X_\pi}))$ is invertible for $m$ sufficiently divisible (this follows from adjunction along $X_\pi\subset X$), hence we can perform adjunction along $E^n\to X_\pi$ even if $X_\pi$ is not normal (see \cite[Definition 4.2]{Kollar_Singularities_of_the_minimal_model_program}). The crucial point is that $E^n$ is a curve, so the points of $E^n$ where $E^n\to D\cap X_\pi$ is not an isomorphism, are contained in $\Supp\Theta$ by \cite[Proposition 4.5.(1)]{Kollar_Singularities_of_the_minimal_model_program}. Thus the the proof of \cite[4.1.1]{Posva_Gluing_for_surfaces_and_threefolds} is also valid in this situation: the finite closed subset $\Sigma:=\Supp\Theta$ is $\sigma$-invariant and $E^n\setminus\Sigma\to E\setminus n(\Sigma)$ is an isomorphism. Thus we obtain finiteness.

Flatness of $X/R(\tau)$ over $\Spec R$ follows from \cite[III.9.7]{Hartshorne_Algebraic_Geometry}.

It remains to show that $X/R(\tau)$ is demi-normal and proper over $\Spec R$. This is shown as in \cite[3.4.1]{Posva_Gluing_for_surfaces_and_threefolds} (which is formulated for schemes over a field, but the proof also works over a DVR).
\end{proof}

\begin{remark}\label{remark:what_do_we_need_for_the_gluing}
More generally, the proof of \autoref{proposition:gluing_exists} applies to $(X,D+\Delta,\tau)$ as soon as we have a gluing theorems for stable lc varieties of dimension $\dim X-1$ above the residue field $k(\pi)$ of $R$, and that the special fiber $X_\pi$ is $S_2$.
\end{remark}

\subsection{Descent of the log canonical sheaf}

We show that in the situation of \autoref{proposition:gluing_exists}, the log canonical $\bQ$-Cartier divisor descends to the quotient. We can actually show it in any dimension, as soon as the quotient exists:

\begin{proposition}\label{proposition:canonical_sheaf_descends}
Let $(Y,\Delta_Y)\to \Spec R$ be a demi-normal flat $\Spec R$-scheme, with induced normalization $(X,\Delta+D,\tau)$. If $(X,D+\Delta+X_\pi)$ is lc and $\Diff_{D^n}(\Delta)$ is $\tau$-invariant, then $K_Y+\Delta_Y$ is $\bQ$-Cartier.
\end{proposition}
\begin{proof}
First base-change over the generic point of $\Spec R$. Then \cite[Theorem 5.38]{Kollar_Singularities_of_the_minimal_model_program} shows that $K_Y+\Delta_Y$ is $\bQ$-Cartier on the generic fiber.

So the closed locus where $K_Y+\Delta_Y$ is not $\bQ$-Cartier, if not empty, is contained in the special fiber $Y_\pi$. By \autoref{lemma:lc_center_on_families} and the fact that $(X,D+\Delta)\to (Y,\Delta_Y)$ is crepant, we see that no lc center of $(Y,\Delta_Y)$ is contained in $Y_\pi$. Thus we may apply the first part of the proof of \cite[Theorem 5.38]{Kollar_Singularities_of_the_minimal_model_program}, which is valid in our setting, and conclude that $K_Y+\Delta_Y$ is also $\bQ$-Cartier along the special fiber.
\end{proof}

\subsection{Proof of the gluing theorem}
\autoref{theorem:gluing_for_surface_families} is an immediate combination of \autoref{proposition:gluing_exists} and \autoref{proposition:canonical_sheaf_descends}. Here are some details:

\begin{proof}[Proof of \autoref{theorem:gluing_for_surface_families}]
If $(X,\Delta)$ is as in the left-hand side of \autoref{theorem:gluing_for_surface_families}, we claim its normalization is a triplet $(\bar{X},\bar{D}+\bar{\Delta},\tau)$ as on the right-hand side. We indeed have a generically fixed point free involution $\tau$: this follows from \autoref{lemma:involution_from_normalization} since by \autoref{lemma:intersection_has_expected_dim} the generic points of $\bar{D}$ have characteristic zero residue fields. The other properties clearly hold.

Let $(\bar{X},\bar{D}+\bar{\Delta},\tau)$ be as in the right-hand side of \autoref{theorem:gluing_for_surface_families}. By \autoref{proposition:gluing_exists} the quotient $X:=\bar{X}/R(\tau)$ exists, it is a demi-normal scheme flat and proper over $\Spec R$. By \autoref{proposition:canonical_sheaf_descends} we obtain that $(X,\Delta)$ is slc.

This defines a map in the opposite direction as the normalization. They are inverse to each other by \cite[Proposition 5.3]{Kollar_Singularities_of_the_minimal_model_program} and the fact that $\bar{X}$ is the normalization of $X$.
\end{proof}

\section{Fibers of the quotient}\label{section:fibers_of_quotient}

In the previous section, we have shown that the quotient of the family exists in several situations. In this section, we study the fibers of the quotient whenever it exists. 

\begin{notation}\label{notation:slc_family_II}
We let:
	\begin{itemize}
		\item $R$ be an excellent DVR (of mixed or equi-characteristic); the generic point of $\Spec R$ is denoted by $\eta$.
		\item $Y\to \Spec R$ be a flat separated morphism of finite type from a demi-normal scheme $Y$ (we do not assume any properness property relatively to $\Spec R$).
		\item $p\colon (X,D,\tau)\to Y$ be the normalization morphism.
		\item Then $D$ is reduced of pure codimension one and we let $\bar{D}\to D$ be its normalization.
	\end{itemize}
\end{notation}

Since $\tau$ is an $R$-involution on $\bar{D}$, by \autoref{lemma:involution_respects_fibration} it restricts to an involution $\tau_\pi$ of $\bar{D}_\pi$ and $\tau_\eta$ of $\bar{D}_\eta$. It is clear that $R_{X_\pi}(\tau_\pi)=R_X(\tau)|_{X_\pi}$ and that $R_{X_\eta}(\tau_\eta)=R_X(\tau)|_{X_\eta}$. Since $R(\tau)\rightrightarrows X$ is finite, it follows that $R(\tau_\pi)\rightrightarrows X\pi$ and $R(\tau_\eta)\rightrightarrows X_\eta$ are finite.

We do not assume systematically that $Y$ has an slc structure. When we do, we let $\Delta_Y$ be the boundary on $Y$, and $\Delta$ its strict transform on $X$.

\subsection{Commutativity of fibers and quotients}

We investigate to which extent the quotients $X_\eta/R(\tau_\eta)$ and $X_\pi/R(\tau_\pi)$ are comparable to the fibers $Y_F$ and $Y_\pi$.

\begin{lemma}\label{lemma:commutativity_for_generic_fiber}
$Y_\eta=X_\eta/R(\tau_\eta)$.
\end{lemma}
\begin{proof}
This follows immediately from \cite[Corollary 9.11]{Kollar_Singularities_of_the_minimal_model_program} since $Y_\eta\to Y$ is flat.
\end{proof}

\begin{lemma}\label{lemma:finite_univ_homeo_on_special_fibers}
The quotient $Z:=X_\pi/R(\tau_\pi)$ exists, and $X_\pi\to Y_\pi$ factors through a finite birational universal homeomorphism $Z\to Y_\pi$.
\end{lemma}
\begin{proof}
The quotient exists as a scheme by \cite[Theorem 6, Corollary 48]{Kollar_Quotients_by_finite_equivalence_relations}. By the universal property of the quotient, $X_\pi\to Y_\pi$ factors through a morphism $q\colon Z\to Y_\pi$. 
\begin{equation}\label{eqn:quotient_diagram}
		\begin{tikzcd}
		X_\pi\arrow[r, hook]\arrow[d] & X\arrow[dd, "p"] \\
		Z\arrow[d, "q"] & \\
		Y_\pi \arrow[r, hook] & Y
		\end{tikzcd}
\end{equation}
Using \cite[Definition 9.2]{Kollar_Singularities_of_the_minimal_model_program} for $X_\pi\to Z$ and $X\to Y$, we see that $q(\Spec K)\colon Z(\Spec K)\to Y_\pi(\Spec Z)$ is a bijection for every geometric point $\Spec K\to \Spec k(\pi)$. Thus $q$ is a universal homeomorphism by \cite[3.5.3-5]{EGA_I}. It is finite birational since $X_\pi\to Y_\pi$ is so.
\end{proof}

Let us study more precisely the morphism $Z\to Y_\pi$. The question is flat-local on $Y$ by \cite[Corollary 9.11]{Kollar_Singularities_of_the_minimal_model_program}, so in particular we may assume that every scheme appearing in \autoref{eqn:quotient_diagram} is affine, and work with sections of structural sheaves as if they were global sections.

We introduce the following sub-sheaves:
		$$\sO_X^+:=\{s\in \sO_X \ : \ s|_{\bar{D}} \text{ is }\tau\text{-invariant}\}, \quad \sO_D^+:=\{t\in \sO_D\ : \ t|_{\bar{D}} \text{ is }\tau\text{-invariant}\}$$
and
		$$\sO_{X_\pi}^+:=\{s\in \sO_X\ : \ s|_{\bar{D}_\pi} \text{ is }\tau_\pi\text{-invariant}\}, \quad \sO_{D_\pi}^+:=\{t\in \sO_{D_\pi}\ : \ t|_{\bar{D}_\pi} \text{ is }\tau_\pi\text{-invariant}\}.$$
They fit into the following commutative diagram:
		\begin{equation}\label{eqn:+_diagram}
			\begin{tikzcd}
			\sO_{D_\pi}^+ && \sO_D^+ \arrow[ll] \\
			\sO_{X_\pi}^+\arrow[u] && \sO_X^+ \arrow[ll]\arrow[u]
			\end{tikzcd}
		\end{equation}

\begin{claim}
The quotient $W(\pi):=D_\pi/R_{D_\pi}(\tau_\pi)$ exists as a scheme, and $\sO_{W(\pi)}=\sO_{D_\pi}^+$.
\end{claim}
\begin{proof}\renewcommand{\qedsymbol}{$\lozenge$}
The quotient exists as a scheme by \cite[Lemma 9.10]{Kollar_Singularities_of_the_minimal_model_program} applied to $D_\pi\to Y_\pi$, and the second assertion follows from the same reference.
\end{proof}

\begin{claim}
$\sO_Z=\sO_{X_\pi}^+$.
\end{claim}
\begin{proof}\renewcommand{\qedsymbol}{$\lozenge$}
Indeed, by \cite[Proposition 25]{Kollar_Quotients_by_finite_equivalence_relations} the diagram
		$$\begin{tikzcd}
		D_\pi \arrow[d] \arrow[r,hook] & X_\pi\arrow[d] \\
		W(\pi)=D_\pi/R_{D_\pi}(\tau_\pi) \arrow[r, hook]& Z
		\end{tikzcd}$$
is a universal push-out. This implies that
		$$\sO_Z=\sO_{X_\pi}\times_{\sO_{D_\pi}}\sO_{W(\pi)}=\sO_{X_\pi}\times_{\sO_{D_\pi}}\sO_{D_\pi}^+=\sO_{X_\pi}^+$$
as claimed.
\end{proof}

\begin{proposition}\label{proposition:equivalent_cond_for_commutativity}
$Z\to Y_\pi$ is an isomorphism if and only if the restriction map $\sO_X^+\to \sO_{X_\pi}^+$ appearing in \autoref{eqn:+_diagram} is surjective.
\end{proposition}
\begin{proof}
By \cite[Lemma 9.10]{Kollar_Singularities_of_the_minimal_model_program} we have $\sO_Y=\sO_X^+$. Combining this with the previous claim and the commutative diagram \autoref{eqn:quotient_diagram} we obtain
		$$\begin{tikzcd}
		\sO_{X_\pi}^+ & \sO_X^+\arrow[l, "\alpha" above] \\
		\sO_Z\arrow[u, "="]&&\\
		\sO_{Y_\pi}\arrow[u, hook] & \sO_Y\arrow[l, twoheadrightarrow]\arrow[uu, "="right]
		\end{tikzcd}$$
It is easy to see that $\alpha$ is surjective if and only if $\sO_{Y_\pi}\hookrightarrow \sO_Z$ is bijective.
\end{proof}

\begin{proposition}\label{proposition:some_minimal_conditions_for_commutativity}
Assume that: 
	\begin{enumerate}
		\item $\tau(\sO_D)\subseteq \sO_D$,
		\item $\sO_{D_\pi}\to \sO_{\bar{D}_\pi}$ is injective, 
		\item $D\to \Spec R$ is flat, and
		\item $2\in R^\times$ (equivalently, $\Char k(\pi)\neq 2)$.
	\end{enumerate}
	Then $Z=Y_\pi$.
\end{proposition}
\begin{proof}
We prove that $\sO_X^+\to \sO_{X_\pi}^+$ is surjective. Take $s\in \sO_{X_\pi}^+$ and any lift $t\in \sO_X$. Write $v=t|_D$. By hypothesis $\tau(v)\in\sO_D$ and by construction $\tau(v)-v$ vanishes when restricted to $\bar{D}_\pi$. Since $\sO_{D_\pi}\to \sO_{\bar{D}_\pi}$ is injective, we see that $v-\tau(v)$ belongs to the ideal $\pi\cdot \sO_{D}$. Thus we can write $\tau(v)=v+a\pi$ for some $a\in\sO_D$. Since $\tau$ is an $R$-involution,
		$$v=\tau^{\circ 2}(v)=v+(a+\tau(a))\pi.$$
By flatness $\pi$ is not a zero-divisor in $\sO_D$, so we have $a=-\tau(a)$. Let $b\in \sO_X$ be any lift of $a$. Since $2\in \sO_X$ is invertible we can form the element $t':=t+\frac{b}{2}\pi$. Then $t'\in\sO_X^+$ and $t'|_{X_\pi}=t|_{X_\pi}=s$, as desired.
\end{proof}

\subsubsection{The case of locally stable families.}\label{section:commutativity_in_slc_case}
We are mainly interested in the case where $(Y,\Delta_Y+Y_\pi)$ is slc for some divisor $\Delta_Y$. In this case the conditions of \autoref{proposition:some_minimal_conditions_for_commutativity} 
are not necessarily met. To wit, consider the following example.

\begin{example}
Let $Y=\Spec R[x,y,z]/(xyz)$. Since $Y$ is an hypersurface, it is Cohen--Macaulay and Gorenstein. Then its normalization $X$ is the union of three copies of $\bA^2_R$, with the morphism $X\to Y$ given by:
		$$\begin{matrix}
		R[x,y,z]/(xyz) & \longrightarrow & R[u_1,u_2] & \oplus & R[v_1,v_2] & \oplus & R[w_1,w_2] \\
		x & \mapsto & u_1&\oplus & v_2 &\oplus &0 \\
		y & \mapsto & u_2&\oplus &0&\oplus & w_1 \\
		z & \mapsto & 0&\oplus & v_1&\oplus &v_2
		\end{matrix}$$
The conductor $D$ is given by $V(u_1u_2)\sqcup V(v_1v_2)\sqcup V(w_1w_2)\subset X$. Notice that $(Y,Y_\pi)$ is slc: for it is easily checked using inversion of adjunction that $(X,D+X_\pi)$ is lc. 

The involution $\tau$ on $\bar{D}$ is given by three isomorphisms of lines, namely
		$$\tau=\left[(\bA^1_{u_1}\cong\bA^1_{v_2},\ u_1\mapsto v_2),\quad 
		(\bA^1_{v_1}\cong\bA^1_{w_2},\ v_1\mapsto w_2),\quad
		(\bA^1_{u_2}\cong\bA^1_{w_1},\ u_2\mapsto w_1)\right].$$
The picture on the special fiber is exactly the same, except that $R$ is replaced with its residue field $k(\pi)$.

The involution $\tau$ does not descend to $D$, since otherwise the three origins would belong to the same orbit. So \autoref{proposition:some_minimal_conditions_for_commutativity} does not apply.

On the other hand, $\sO_X^+$ is the set of $f(s(u_1,v_2),s'(v_1,w_2),s''(u_2,w_1))$ where $f\in R[X,Y,Z]$ and $s,s',s''$ run through the symmetric polynomials in two variables. Similarly for $\sO_{X_\pi}^+$, except that we take $f\in k(\pi)[X,Y,Z]$. In particular we see that $\sO_X^+\to \sO_{X_\pi}^+$ is surjective, so the conclusion of \autoref{proposition:some_minimal_conditions_for_commutativity} holds nonetheless.
\end{example}

An easy application of \autoref{proposition:some_minimal_conditions_for_commutativity} is the following:
\begin{proposition}
Suppose that $\Char k(\pi)\neq 2$ and that $(Y,\Delta_Y+Y_\pi)$ is slc. Then:
	\begin{enumerate}
		\item $Z\to Y_\pi$ is an isomorphism in codimension one.
		\item If $Y_\pi$ is $S_2$, then $Z\to Y_\pi$ is an isomorphism.
		\item If $D$ is normal, then $Z\to Y_\pi$ is an isomorphism.
	\end{enumerate}
\end{proposition}
\begin{proof}
Let $(X,\Delta+D+X_\pi)$ be the normalization. By \cite[Corollary 2.32]{Kollar_Singularities_of_the_minimal_model_program} we see that $D$ is normal in codimension one along $D_\pi$, thus $\tau(\sO_D)\subseteq \sO_D$ holds in a neighbourhood of the generic points of $D_\pi$. 

We consider the pullback morphism $\sO_{D_\pi}\to \sO_{\bar{D}_\pi}$. Recall that $\bar{D}_\pi$ is reduced (\autoref{lemma:special_fiber_is_reduced}) and $\bar{D}_\pi\to D_\pi$ is dominant. Thus $\sO_{D_\pi}\to \sO_{\bar{D}_\pi}$ is injective if and only if $D_\pi$ is reduced. By \cite[Corollary 2.32]{Kollar_Singularities_of_the_minimal_model_program}, we know that $D_\pi$ is generically reduced. If $D$ is normal, then $D_\pi$ is $S_1$ and thus reduced everywhere.

Hence we deduce that there is an open subset $U\subset Y$ such that $X_U$ contains every codimension one point of $X_\pi$ and such that the conditions of \autoref{proposition:some_minimal_conditions_for_commutativity} are satisfied for the normalization $(X_U,D_U,\tau_U)\to U$. It follows that $Z_U\to Y_\pi\cap U$ is an isomorphism. This proves the first point, and the second point follows easily. If $D$ is normal we can take $U=Y$, thus obtaining the third point.
\end{proof}

In the surface case, a finer analysis of the singularities of $D$ yields a stronger statement under mild hypothesis:
\begin{theorem}\label{thm:commutativity_in_slc_case}
Suppose that $k(\pi)$ is perfect of characteristic $p\neq 2$ and that $(Y,\Delta_Y+Y_\pi)$ is slc of dimension $3$ and quasi-projective over $\Spec R$. If the scheme-theoretic intersection $D_\pi$ is reduced, then $Z\to Y_\pi$ is an isomorphism. 
\end{theorem}

Before proving this theorem, let us comment on the reducedness of $D_\pi$. It is is the scheme-theoretic intersection of two lc centers of $(X,\Delta+D+X_\pi)$. In characteristic zero, intersections of lc centers are reduced \cite[Theorem 7.8]{Kollar_Singularities_of_the_minimal_model_program}. In mixed or equi-characteristic, we can prove the following.

\begin{proposition}\label{proposition:intersection_of_lc_centers}
If $(Y,\Delta_Y+Y_\pi)$ is an slc threefold pair of finite type over $\Spec R$, and $k(\pi)$ is perfect of characteristic $p>5$, then $D_\pi$ is reduced. 
\end{proposition}
\begin{proof}
First assume that $(X,\Delta+D+X_\pi)$ is dlt. Then the result holds by \cite[Theorem 19 and subsequent paragraph]{Bernasconi_Kollar_vanishing_theorems_for_threefolds_in_char_>5}. 

In the general case, let $\varphi\colon (X',\Delta'+D'+E'+X_\pi')\to (X,\Delta+D+X_\pi)$ be a $\bQ$-factorial crepant dlt model. Such a model exists: log resolutions exist for pairs of dimension three over the spectrum of $R$ \cite[2.12]{Bhatt&Co_MMP_for_3folds_in_mixed_char}, and combining \cite[Theorem F]{Bhatt&Co_MMP_for_3folds_in_mixed_char} with the arguments of \cite[1.35-1.36]{Kollar_Singularities_of_the_minimal_model_program} we can run a MMP to produce a model with the desired properties. We claim that $R^1\varphi_*\sO(-D'-X_\pi')=R^1\varphi_*\sO(-D')\otimes \sO(-X_\pi)=0$ along $\Supp(X_\pi)$. This follows from \cite[Proposition 7]{Bernasconi_Kollar_vanishing_theorems_for_threefolds_in_char_>5} applied to $\varphi\colon X'\to X$, the dlt pair $(X',X_\pi'+(\Delta'+E'))$ and the $\bZ$-divisor $-D'$, since:
	\begin{enumerate}
		\item $-D'\sim_{\varphi,\bQ}K_{X'}+X'_\pi+E'+\Delta'$,
		\item $\sO(-D'-mX'_\pi)$ is $S_3$ for every $m\geq 1$ by \cite[Theorem 19]{Bernasconi_Kollar_vanishing_theorems_for_threefolds_in_char_>5},
		\item $-X'_\pi$ is relatively nef over $X$, since it is relatively trivial,
		\item strong Grauert--Riemenschneider vanishing holds for $(X_\pi')^n\to X_\pi$ since this is a birational morphism of excellent surfaces (this can be deduced from \cite[Theorem 10.4]{Kollar_Singularities_of_the_minimal_model_program}).
	\end{enumerate}
	Therefore pushing forward along $\varphi$ the exact sequence
			$$0\to \sO(-D'-X'_\pi)\to \sO_{X'}\to \sO_{D'\cup X'_\pi}\to 0,$$
	we obtain that $\varphi_*\sO_{X'}=\sO_X\to \varphi_*\sO_{D'\cup X'_\pi}$ is surjective in a neighbourhood of $X_\pi$. This map factors through $\sO_{D\cup X_\pi}$, so we deduce that
			$$\varphi_*\sO_{D'\cup X'_\pi}=\sO_{D\cup X_\pi}.$$
	We are now in position to apply the argument of \cite[Theorem 7.8]{Kollar_Singularities_of_the_minimal_model_program} to conclude that $D\cap X_\pi$ is reduced.
\end{proof}

The following lemma will be useful for the proof of \autoref{thm:commutativity_in_slc_case}. If $T$ is a scheme, $t\in H^0(T,\sO_T)$ and $\fp\in T$ is any point, we write $\res_\eta(t)$ for the image of $t$ through the natural map $H^0(T,\sO_T)\to k(\fp)$.

\begin{lemma}\label{lemma:transfert_between_nodes}
Let $X=\Spec A$ be an affine excellent demi-normal scheme with normalization $\bar{X}=\Spec\bar{A}$. Let $\eta\in X$ be a node, and assume that $\eta$ has two preimages $\xi,\xi'\in \bar{X}$. Let $V=\Spec \bar{A}/\fp_\xi$ and $V'=\Spec \bar{A}/\fp_{\xi'}$.  Then:
	\begin{enumerate}
		\item The natural maps $k(\eta)\to k(\xi)$ and $k(\eta)\to k(\xi')$ are isomorphisms. We denote by $\phi\colon k(\xi')\cong k(\xi)$ the induced isomorphism.
		\item Let $v\in \bar{A}$. If $V$ is normal, then the element $\res_\xi(v)-\phi(\res_{\xi'}(v))\in k(\xi)$ extends to a regular function on $V$.
		\item If both $V$ and $V'$ are normal, then $\phi$ extends to an isomorphism $V\cong V'$.
	\end{enumerate}
\end{lemma}
\begin{proof}
The first point follows from \autoref{lemma:properties_of_nodes}. To prove the other two, set $S=\Spec A/\fp_\eta$ and let $\Gamma$ be the main component of the fiber product $V\times_S V'$, equipped with its reduced structure. Its generic point is $(\xi,\xi')$ and its function field is $k(\Gamma)=k(\xi)\otimes_{k(\eta)}k(\xi')$. By construction the projection $\pr_1\colon \Gamma\to V$ is finite, thus closed; since its image contains the generic point $\xi$ of $V$, it is surjective. By the first point we see that $\pr_1$ induces an equality of functions fields, so it is birational. Similarly $\pr_2\colon \Gamma\to V'$ is finite surjective birational. 

Under the identification $k(\xi)\otimes_{k(\eta)}k(\xi')=k(\xi)$ induced by $\phi$, we have
		$$\pr_1^*(v)-\pr_2^*(v)|_{k(\Gamma)}=\res_\xi(v)-\phi(\res_{\xi'}(v)).$$
Assume that $V$ is normal. Then by Zariski's Main Theorem the finite surjective birational morphism $\pr_1\colon \Gamma\to V$ is an isomorphism. Thus the rational function $\res_\xi(v)-\phi(\res_{\xi'}(v))$ extends to a regular function on $V$, namely the section corresponding under $\pr_1$ to $\pr_1^*(v)-\pr_2^*(v)\in H^0(\Gamma,\sO_\Gamma)$.

Finally, if $V$ and $V'$ are normal, then the composition $\pr_2\circ\pr_1^{-1}\colon V\to V'$ gives the isomorphism extending $\phi$.
\end{proof}

\begin{proof}[Proof of \autoref{thm:commutativity_in_slc_case}]
We assume from now on that $D_\pi$ is reduced. The first important observation is that the divisor $D$ has mild singularities:

\begin{claim}\label{claim:structure_of_D}
The surface $D$ is demi-normal with singular codimension one points mapping to the generic point $\eta$ of $\Spec R$.
\end{claim}
\begin{proof}\renewcommand{\qedsymbol}{$\lozenge$}
By hypothesis $D_\pi$ is $S_1$. Since it is a hypersurface of $D$ and since $\pi$ is not a zero-divisor in $\sO_D$, we deduce that $\depth \sO_{D,x}\geq \min\{2,\dim\sO_{D,x}\}$ for every point $x$ of the special fiber. The local rings at the points of the generic fiber have dimension at most one and are reduced. Thus we obtain that $D$ is $S_2$.

Let $\xi\in D$ be a codimension one point. Then $\xi\in X$ is a codimension two point, and by \cite[Theorem 2.31]{Kollar_Singularities_of_the_minimal_model_program} we see that $D$ is at worse nodal at $\xi$. Thus $D$ is demi-normal. Assume that $\xi$ is the generic point of an irreducible component of $D_\pi$. Then $\xi$ belongs to the intersection $X_\pi\cap D$, and by \cite[Theorem 2.31]{Kollar_Singularities_of_the_minimal_model_program} again we see that $\xi$ is a regular point of $D$. Therefore, if $\xi$ is a singular point of $D$, it maps to the generic point of $\Spec R$.
\end{proof}

\begin{claim}\label{claim:nodes_are_disjoint_normal}
Write $\lfloor \Diff_{\bar{D}_\eta}(\Delta) \rfloor =\sum_\alpha p_\alpha$, where each $p_\alpha\in \bar{D}$ is a codimension one point with ideal $\fp_\alpha\subset \sO_{\bar{D}}$. Then:
	\begin{enumerate}
		\item the collection $\{p_\alpha\}$ contains all the preimages of the nodes of $D$, and
		\item the $V(p_\alpha)=\Spec\sO_{\bar{D}}/\fp_\alpha$ are normal and pairwise disjoint subschemes of $\bar{D}$.
	\end{enumerate}
\end{claim}
\begin{proof}\renewcommand{\qedsymbol}{$\lozenge$}
Recall that $(\bar{D},\Diff_{\bar{D}}(\Delta)+\bar{D}_\pi)$ is an lc surface by hypothesis. If $m\in\bar{D}$ is a preimage of a node of $D$, a local calculation shows that the divisors $V(m)$ appear with coefficient one in $\Diff_{\bar{D}}(\Delta)$, see for example \cite[Theorem 2.31.(2)]{Kollar_Singularities_of_the_minimal_model_program}. If $z\in V(p_\alpha)\cap V(p_\beta)$ then $z$ belongs to the special fiber $\bar{D}_\pi$. However this contradicts \cite[Theorem 2.31.(2)]{Kollar_Singularities_of_the_minimal_model_program}. So the $V(p_\alpha)$ are pairwise disjoint.

By \cite[Theorem 2.31.(2)]{Kollar_Singularities_of_the_minimal_model_program} again, the intersection $V(p_\alpha)\cap \bar{D}_\pi$ is regular. We deduce that the one-dimensional reduced scheme $V(p_\alpha)$ is normal around its special fiber. It is also normal along its generic fiber. Therefore it is normal everywhere.
\end{proof}

By \cite[Proposition 25]{Kollar_Quotients_by_finite_equivalence_relations} the diagram
		$$\begin{tikzcd}
		D\arrow[r, hook]\arrow[d] & X\arrow[d,"p"] \\
		W=D/R_D(\tau)\arrow[r, hook] & Y 
		\end{tikzcd}$$
is a universal pushout. Therefore the diagram
		\begin{equation}\label{eqn:pushout_over_special_point}
		\begin{tikzcd}
		D_\pi\arrow[r, hook]\arrow[d] & X_\pi\arrow[d,"p_\pi"] \\
		W_\pi\arrow[r, hook] & Y_\pi 
		\end{tikzcd}
		\end{equation}
is also a pushout.

\begin{claim}\label{claim:pushout_over_special_point}
Every scheme appearing in \autoref{eqn:pushout_over_special_point} is reduced.
\end{claim}
\begin{proof}\renewcommand{\qedsymbol}{$\lozenge$}
By assumption $D_\pi$ is reduced, and $X_\pi$ is reduced by \autoref{lemma:adjunction_route_along_special_fiber}. Assume for the moment that $Y_\pi$ is reduced. If $0\neq s\in \nil(\sO_{W_\pi})$, then by the pushout property $(s,0)\in \sO_{W_\pi}\times\sO_{X_\pi}$ gives a non-zero nilpotent element of $Y_\pi$, a contradiction. Thus $W_\pi$ is reduced.

To show that $\sO_{Y_\pi}$ is reduced, consider the exact sequence
		$$0\to \sO_Y\to p_*\sO_X\to \sQ\to 0$$
and tensor it by $R/\pi R$ over $R$. If $\Tor^R_1(R/\pi R, \sQ)=0$ then $\sO_{Y_\pi}$ embeds into $(p_\pi)_*\sO_{X_\pi}$. Since the latter is a sheaf of reduced algebras, this would entail that $\sO_{Y_\pi}$ is reduced as well. Now assume that $\Tor^R_1(R/\pi R, \sQ)\neq 0$. Then there is $q\in p_*\sO_X$ such that $q\notin \sO_Y$ but $\pi q\in \sO_Y$. But then $q\pi|_{\bar{D}}$ is $\tau$-invariant. Since $\pi|_{\bar{D}}$ is $\tau$-invariant and $\bar{D}$ is flat over $\Spec R$, it must hold that $q|_{\bar{D}}$ is $\tau$-invariant. Therefore $q\in \sO_Y$, which is a contradiction. Hence $ \Tor^R_1(R/\pi R, \sQ)=0$.
\end{proof}

We have seen in \autoref{lemma:finite_univ_homeo_on_special_fibers} that $Z\to Y_\pi$ is a finite universal homeomorphism. Hence if $X_\pi\to Y_\pi$ is weakly normal, then $Z\to Y_\pi$ is an isomorphism (see \autoref{section:weak_normality}). Therefore, by \autoref{claim:pushout_over_special_point} and \autoref{lemma:pushout_and_weak_normality}, we obtain:

\begin{observation}
To prove \autoref{thm:commutativity_in_slc_case} it is sufficient to prove that $D_\pi\to W_\pi$ is weakly normal.
\end{observation}

The situation can be slightly simplified:

\begin{claim}\label{claim:nodes_are_split}
We may (and will) assume that:
	\begin{enumerate}
		\item $k(\pi)$ is algebraically closed,
		\item any node $n\in D$ has two preimages $m,m'\in \bar{D}$, and
		\item we have canonical equalities $k(\eta)=k(m)=k(n)=k(m')$.
	\end{enumerate}	 
\end{claim}
\begin{proof}\renewcommand{\qedsymbol}{$\lozenge$}
We are studying whether $D_\pi\to W_\pi$ is weakly normal. This property descends faithfully flat covers \cite[Corollary to Proposition 1]{Yanagihara_Weakly_normal_ring_extensions}. Moreover \'{e}tale base-changes commute with normalizations and preserve slc singularities. Thus we can base-change along the strict henselization $\Spec R^\text{sh}\to \Spec R$ and assume that $k(\pi)$ is algebraically closed.

Let $n\in D$ be a node, and let $\bar{n}\in\bar{D}$ be a preimage. I claim that $k(n)=k(\bar{n})$. To show this, by \autoref{prop:compactification_slc_surface} we may assume that $D$ is proper over $\Spec R$. By \autoref{claim:structure_of_D} both $n$ and $\bar{n}$ maps to the generic point $\eta$ of $\Spec(R)$. Then $V=V(\bar{n})\subset \bar{D}$ is one-dimensional, it is normal by \autoref{claim:nodes_are_disjoint_normal}, and the morphism $V\to\Spec R$ is flat and proper. Therefore $V\to \Spec R$ is finite, and by flatness we have
		$$\length_{k(\eta)}V_{\eta}=\length_{k(\pi)}V_{\pi}.$$
Since $V_\eta$ is the spectrum of the field $k(\bar{n})$, by the valuative criterion of properness we deduce that the special fiber $V_\pi$ is connected. We have seen in the proof of \autoref{claim:nodes_are_disjoint_normal} that $V_{\pi}$ is reduced. 
Therefore $V_\pi$ is the spectrum of a finite field extension of $k(\pi)$. By the assumption on $k(\pi)$ we obtain that $\length_{k(\pi)}V_{\pi}=1$, and therefore $V_{\eta}$ is a $k(\eta)$-point. This implies that $k(n)=k(\bar{n})$. Now if $\bar{n}$ is the unique preimage of $n$, then $k(n)\subset k(\bar{n})$ is a field extension of degree $2$ (see \autoref{lemma:properties_of_nodes}), which is a contradiction. Thus $n$ has two preimages $\bar{n},\bar{n}'\in\bar{D}$, with equalities of residue fields $k(\eta)=k(n)=k(\bar{n})=k(\bar{n}')$.
\end{proof}

By \cite[Corollary to Theorem 1]{Yanagihara_Weakly_normal_ring_extensions}, weak normality of $D_\pi\to W_\pi$ is equivalent to the following property: if $x\in \sO_{D_\pi}$ is such that $x^p\in \sO_{W_\pi}$, then $x\in\sO_{W_\pi}$ already. 

Let $v\in \sO_D$ be any lift of $x$. Arguing as in the proof of \autoref{proposition:some_minimal_conditions_for_commutativity}, we find that $\tau(v)=v+b\pi$, where $b\in \sO_{\bar{D}}$ is such that $\tau(b)=-b$. Replacing $v$ by $v+\frac{b}{2}\pi$, we may assume that $v$ is $\tau$-invariant, but at the cost that $v$ only belongs to $\sO_{\bar{D}}$. However we retain the property that the reduction of $v$ modulo $\pi$ belongs to $\sO_{D_\pi}$.

The strategy is now the following: building on \autoref{corollary:descend_section_by_gluing}, descend $v+c\pi$ to $\sO_D$ for a well-chosen $c\in \sO_{\bar{D}}^+$. If we manage to do so, we are done, for then $v+c\pi\in \sO_D^+=\sO_W$ restricts to $x\in \sO_{W_\pi}$.

Thus we reduce to prove the following claim:
\begin{claim}\label{claim:correction_term_for_descent}
Under the assumptions of \autoref{claim:nodes_are_split}, given $v\in \sO_{\bar{D}}^+$ such that $v|_{\bar{D}_\pi}\in \sO_{D_\pi}$, there exists $c\in\sO_{\bar{D}}^+$ such that $v+c\pi\in\sO_D$.
\end{claim}

To fix the ideas, let us consider the two simplest cases. Recall that $\eta$ is the generic point of $\Spec R$.

\begin{example}
Assume that $(X,D+\Delta)$ is plt on the generic fiber. Then $D_\eta=\bar{D}_\eta$, and in this situation the condition of \autoref{corollary:descend_section_by_gluing} holds automatically. So $v$ descends to $\sO_D$.
\end{example}

\begin{example}\label{example:correction_over_nodes}
Assume that $D_\eta$ has a single node $n$ and $X_\eta$ is regular along $D_\eta$. This node is not contained in the support of $\Delta$, so let us assume for simplicity that $\Delta=0$. Then $\Diff_{\bar{D}_\eta}(0)=m+m'$: this divisor is preserved by $\tau_\eta$.

	Denote by $\phi\colon k(m')\overset{\sim}{\to}k(m)$ the isomorphism on residue fields given by \autoref{lemma:transfert_between_nodes}. We say that \emph{$v$ has the same residues at $m$ and $m'$} if $\fr_n(v):=\res_m(v)-\phi(\res_{m'}(v))=0$. 
		\begin{enumerate}
			\item If $v$ has the same residues at $m$ and $m'$ (this happens for example if $\tau(m)=m'$, since $v$ is $\tau$-invariant) then $v$ descends to $\sO_D$ by \autoref{corollary:descend_section_by_gluing}.
			
			\item Assume that $v$ has not the same residues at $m$ and $m'$. Then $\fr_n(v)$ is a rational function on $V(m)$. As in \autoref{lemma:transfert_between_nodes}, if $\Gamma\subset V(m)\times_{V(n)}V(m')$ is the main component with reduced structure then
			$$\pr_1^*(v)-\pr_2^*(v)=\fr_n(v) \quad \text{in}\quad k(\Gamma)=k(m)\otimes_{k(n)}k(m').$$
			Since $V(m)$ is normal by \autoref{claim:nodes_are_disjoint_normal}, the projection $\pr_1\colon \Gamma\to V(m)$ is an isomorphism. We deduce that $\fr_n(v)$ is regular on $V(m)$. Moreover $\pr_1^*(v)-\pr_2^*(v)$ vanishes modulo $\pi$, since $v|_{\bar{D}_\pi}$ belongs to $\sO_{D_\pi}$. So we obtain that $\fr_n(v)=\bar{c}\pi$ for some $\bar{c}\in \sO_{V(m)}$.
		
		 We have $\tau(m)=m$ and $\tau(m')=m'$. The isomorphism $\phi\colon k(m')\cong k(m)$ is $\tau$-equivariant (this follows for example from \autoref{claim:nodes_are_split}), and since $v$ is $\tau$-invariant we deduce that $\bar{c}\pi$ is $\tau$-invariant. Therefore $\bar{c}$ is also $\tau$-invariant. Now consider the restriction map 
				$$\varphi\colon \sO_{\bar{D}}\longrightarrow \sO_{V(m)}\oplus \sO_{V(m')}.$$
		By \autoref{claim:nodes_are_disjoint_normal} and \autoref{lemma:equivariant_Chinese_remainder_thm} below there is $c\in \sO_{\bar{D}}^+$ such that $\varphi(c)=(\bar{c},0)$. By construction $v-\bar{c}\pi$ is $\tau$-invariant and descends to $\sO_D$ thanks to \autoref{corollary:descend_section_by_gluing}.
		\end{enumerate}
\end{example}

In the second example we used the following lemma:

\begin{lemma}\label{lemma:equivariant_Chinese_remainder_thm}
Let $A$ be a Noetherian ring acted on by a finite group $G$, such that $|G|\in A^\times$. Let $\{\fp_1,\dots,\fp_n\}$ be a $G$-invariant set of prime ideals, with the property that $\fp_i+\fp_j=A$ for any $i\neq j$. Then the natural map
		$$\varphi^G\colon A^G\longrightarrow \left( \bigoplus_{i=1}^n A/\fp_i\right)^G$$
is surjective.
\end{lemma}
\begin{proof}\renewcommand{\qedsymbol}{$\lozenge$}
By the Chinese remainder theorem the natural restriction map
	$$\varphi\colon A\longrightarrow \bigoplus_{i=1}^n A/\fp_i$$
is surjective. It is easily seen to be $G$-equivariant: we denote by $\varphi^G$ the induced map between the $G$-invariant subrings. Take any $G$-invariant element $\bold{a}\in \bigoplus_i A/\fp_i$, and choose a lift $a\in A$. Then $a_\mu:=\frac{1}{|G|}\sum_{g\in G} g(a)$ is $G$-invariant and satisfies $\varphi^G(a_\mu)=\bold{a}$. This shows that $\varphi^G$ is also surjective.
\end{proof}	

Now let us treat the general case: the essential ideas are all contained in \autoref{example:correction_over_nodes} above. Let $n_1,\dots,n_r\in D_\eta$ be the nodes of $D$. Notice that $\Sing(D_\eta)=\{n_1,\dots,n_r\}$ since $D_\eta$ is a nodal curve. Let $m_i,m_i'\in\bar{D}_\eta$ be the two preimages of $n_i$. For any such preimage $m$, we define $V(m)$ as in \autoref{claim:nodes_are_disjoint_normal}. We also let $\phi_i\colon k(m_i)\cong k(m_i')$ be the canonical isomorphism of residue fields over the node $n_i$ (see \autoref{lemma:transfert_between_nodes}).

We can write
		$$\lfloor\Diff_{\bar{D}_\eta}(\Delta)\rfloor=\left(\sum_{i=1}^rm_i+m_i'\right)+\sum_{j=1}^sd_j,$$
	where the points $d_j\in\bar{D}_\eta$ belong to the locus where $\bar{D}_\eta\to D_\eta$ is an isomorphism. Since $\tau$ preserves $\Diff_{\bar{D}}(\Delta)$, on the generic fiber $\tau_\eta$ preserves $\lfloor\Diff_{\bar{D}_\eta}(\Delta)\rfloor$. Therefore the collection $\sP=\{m_i,m_i',d_j\}_{i,j}$ is a $\tau_\eta$-invariant subset of $\bar{D}_\eta$. It is also a union of equivalence classes for the relation given by the fibers of $\bar{D}_\eta\to D_\eta$, because it contains all the $m_i,m_i'$.

	The equivalence relation generated by $\tau_\eta$ and by the fibers of $\bar{D}_\eta\to D_\eta$ gives a partition $\sP=\bigsqcup_l\sP_l$ into classes that are of four possible types (up to relabelling):
	
	
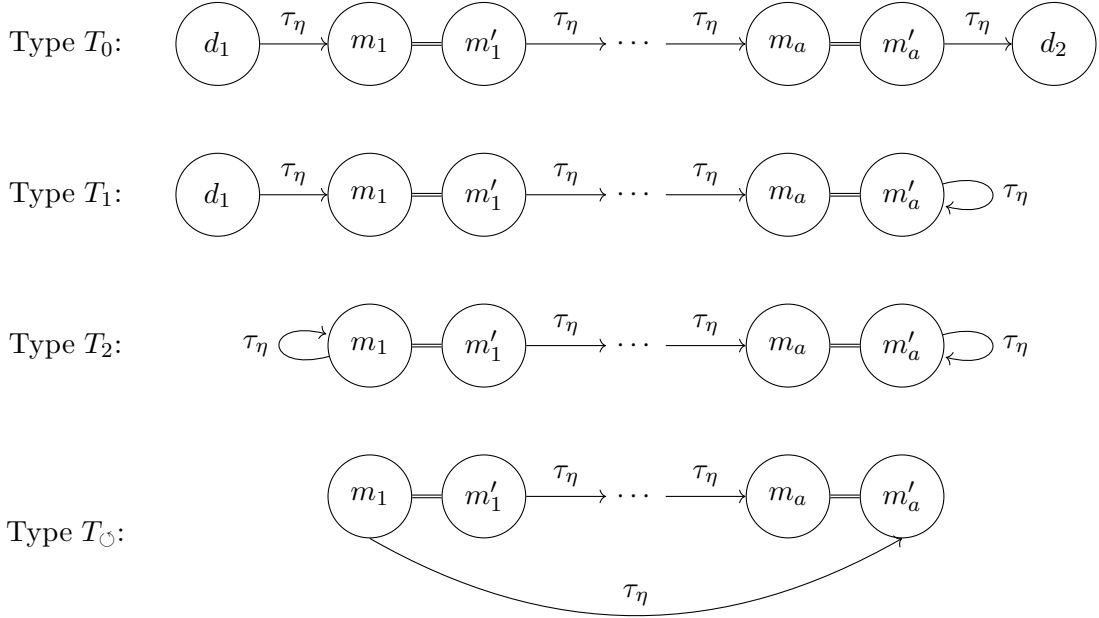
\vspace{0.5cm}
	\begin{center}
	\begin{tikzpicture}
	\node at (0,-1) {Type $T_0$:};
	\node[draw, circle, minimum size=1.1cm] at (2,-1) (d) {$d_{1}$};
	\node[draw, circle, minimum size=1.1cm] at (4,-1) (m1) {$m_{1}$};
	\node[draw, circle,minimum size=1.1cm] at (5.5,-1) (m1') {$m_{1}'$};
	\node at (7.5,-1) (etc) {$\cdots$};
	\node[draw, circle, minimum size=1.1cm] at (9.5,-1) (ma) {$m_{a}$};
	\node[draw, circle, minimum size=1.1cm] at (11,-1) (ma') {$m_{a}'$};
	\node[draw, circle, minimum size=1.1cm] at (13,-1) (d') {$d_{2}$};
	
	\path (d) edge[->, "$\tau_\eta$"] (m1)
		(m1) edge[-, double] (m1')
		(m1') edge[-, ->, "$\tau_\eta$"] (etc)
		(etc) edge[-,->,"$\tau_\eta$"] (ma)
		(ma) edge[-, double] (ma')
		(ma') edge[-,->,"$\tau_\eta$"] (d');

	\node at (0,-3) {Type $T_1$:};
	\node[draw, circle, minimum size=1.1cm] at (2,-3) (db) {$d_{1}$};
	\node[draw, circle, minimum size=1.1cm] at (4,-3) (m1b) {$m_{1}$};
	\node[draw, circle, minimum size=1.1cm] at (5.5,-3) (m1'b) {$m_{1}'$};
	\node at (7.5,-3) (etcb) {$\cdots$};
	\node[draw, circle, minimum size=1.1cm] at (9.5,-3) (mab) {$m_{a}$};
	\node[draw, circle, minimum size=1.1cm] at (11,-3) (ma'b) {$m_{a}'$};
	
	\path (db) edge[->, "$\tau_\eta$"] (m1b)
		(m1b) edge[-, double] (m1'b)
		(m1'b) edge[-, ->, "$\tau_\eta$"] (etcb)
		(etcb) edge[-,->,"$\tau_\eta$"] (mab)
		(mab) edge[-, double] (ma'b)
		(ma'b) edge[-, loop right, "$\tau_\eta$"] (ma'b);
	
	\node at (0,-5) {Type $T_2$:};
	\node[draw, circle, minimum size=1.1cm] at (4,-5) (m1c) {$m_{1}$};
	\node[draw, circle, minimum size=1.1cm] at (5.5,-5) (m1'c) {$m_{1}'$};
	\node at (7.5,-5) (etcc) {$\cdots$};
	\node[draw, circle, minimum size=1.1cm] at (9.5,-5) (mac) {$m_{a}$};
	\node[draw, circle, minimum size=1.1cm] at (11,-5) (ma'c) {$m_{a}'$};
	
	\path (m1c) edge[-, loop left, "$\tau_\eta$"] (m1c)
		(m1c) edge[-, double] (m1'c)
		(m1'c) edge[-, ->, "$\tau_\eta$"] (etcc)
		(etcc) edge[-,->,"$\tau_\eta$"] (mac)
		(mac) edge[-, double] (ma'c)
		(ma'c) edge[-, loop right, "$\tau_\eta$"] (ma'c);
	
	\node at (0, -7.5) {Type $T_\circlearrowleft$:};
	\node[draw, circle, minimum size=1.1cm] at (4,-7) (m1d) {$m_{1}$};
	\node[draw, circle, minimum size=1.1cm] at (5.5,-7) (m1'd) {$m_{1}'$};
	\node at (7.5,-7) (etcd) {$\cdots$};
	\node[draw, circle, minimum size=1.1cm] at (9.5,-7) (mad) {$m_{a}$};
	\node[draw, circle, minimum size=1.1cm] at (11,-7) (ma'd) {$m_{a}'$};
	
	\path (m1d.south) edge[->, bend right, "$\tau_\eta$"] (ma'd.south) 
	(m1d) edge[-, double] (m1'd)
	(m1'd) edge[-, ->, "$\tau_\eta$"] (etcd)
	(etcd) edge[-,->,"$\tau_\eta$"] (mad)
	(mad) edge[-, double] (ma'd);
	\end{tikzpicture}
	\end{center}

	By \autoref{claim:nodes_are_split}, the structural map $k(\eta)\hookrightarrow k(m)$ is an equality for any $m\in\{m_i,m_i'\}$. Thus if $\tau_\eta(m_i')=m_j$ then the induced isomorphism on residue fields $\tau_\eta\colon k(m_i')\cong k(m_j)$ is the unique isomorphism that makes the diagram
		$$\begin{tikzcd}
		k(m_i')\arrow[rr, dotted, "\tau_\eta"]\arrow[dr, "\cong" below left] && k(m_j)\arrow[dl, "\cong"]\\
		& k(\eta) &
		\end{tikzcd}$$ 
	commute.
	
	Now consider $v\in \sO_{\bar{D}}^+$. We write $\res_p(v)$ for the image of $v$ in the residue field of any $p\in \sP$. As in \autoref{example:correction_over_nodes}, for any $i$ we can write
			$$\res_{m_i'}(v)-\phi_i(\res_{m_i}(v))=\gamma_i\pi\quad \text{for some }\gamma_i\in\sO_{V(m_i')}.$$
	In general these differences are non-zero, and we look for a correction term that is $\tau$-invariant. Take an equivalence class $\sP_l\subset \sP$, and say that the preimages of nodes contained in $\sP_l$ are $m_1,m_1',\dots,m_a,m_a'$. Up to relabelling we may assume that 
			$$\tau_\eta(m_1')=m_2,\quad \tau_\eta(m_2')=m_3,\quad \dots \quad \tau_\eta(m_{a-1}')=m_a,$$
		or in picture that

		
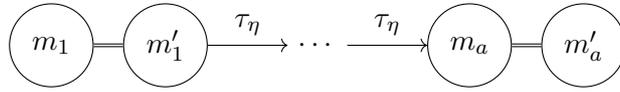
\begin{center}\begin{tikzpicture}
		\node[draw, circle, minimum size=1.1cm] at (4,0) (m1b) {$m_{1}$};
	\node[draw, circle, minimum size=1.1cm] at (5.5,0) (m1'b) {$m_{1}'$};
	\node at (7.5,0) (etcb) {$\cdots$};
	\node[draw, circle, minimum size=1.1cm] at (9.5,0) (mab) {$m_{a}$};
	\node[draw, circle, minimum size=1.1cm] at (11,0) (ma'b) {$m_{a}'$};
	
		\path 
		(m1b) edge[-, double] (m1'b)
		(m1'b) edge[-, ->, "$\tau_\eta$"] (etcb)
		(etcb) edge[-,->,"$\tau_\eta$"] (mab)
		(mab) edge[-, double] (ma'b);
		\end{tikzpicture}
		\end{center}	
		so that the values $\tau_\eta(m_1),\tau_\eta(m_a')$ determine if $\sP_l$ is of type $T_0,T_1,T_2$ or $T_\circlearrowleft$.
	For $i=1,\dots,a$ we define inductively $c_i\in \sO_{V(m_i')}$ by
		\begin{equation}\label{eqn:def_correction_term}
		c_1=\gamma_1,\quad c_{i+1}\pi=\phi_{i+1}(\tau_\eta(c_i\pi))-\gamma_{i+1}\pi \quad (1\leq i\leq a-1).
		\end{equation}
	We think of $c_{i+1}$ as a correction term for the residue of $v$ at the point $m_{i+1}'$. Indeed, by construction we have
			\begin{equation}\label{eqn:correction_term}
			\phi_{i+1}(\res_{m_{i+1}}(v)+\tau_\eta(c_{i+1}\pi))=\res_{m_{i+1}'}(v)+c_{i+1}\pi.
			\end{equation}
	\begin{claim}
	These $c_i\in \sO_{V(m_i')}$ exist. 
	\end{claim}
	\begin{proof}\renewcommand{\qedsymbol}{$\lozenge$}
	We check this inductively on $i$. It is clear for $i=1$. Assuming that $c_i\in\sO_{V(m_i')}$ is constructed, it is sufficient to show that
		$$\phi_{i+1}(\tau_\eta(c_i))-\gamma_{i+1}\in k(m_{i+1}) \text{ is a regular function on }V(m_{i+1}).$$ 
		Since $\gamma_{i+1}$ is a regular function, we reduce to show that $\phi_{i+1}(\tau_\eta(c_i))$ is a regular function on $V(m_{i+1})$. Now $\tau_\eta$ is the extension of $\tau\colon V(m_i')\cong V(m_{i+1})$, and by \autoref{lemma:transfert_between_nodes} $\phi_{i+1}$ extends to an isomorphism $V(m_{i+1})\cong V(m_{i+1}')$. Since $c_i$ is regular on $V(m_i')$ we deduce that $\phi_{i+1}(\tau_\eta(c_i))$ is regular on $V(m_{i+1}')$.
	\end{proof}

	\noindent For each $\sP_l$, we record the above previous construction using a vector that we call $\bold{c}_{\sP_l}$:
		\begin{enumerate}
			\item If $\sP_l$ is of type $T_2$ or $T_\circlearrowleft$ we let
				$$\bold{c}_{\sP_l}=\left(0,c_1\pi,\tau(c_1\pi),\dots, \tau(c_{a-1}\pi), c_a\pi)\right)\in 
			\bigoplus_{i=1}^a\sO_{V(m_i)}\oplus \sO_{V(m_i')}.$$
			\item If $\sP_l$ is of type $T_1$ we let
				$$\bold{c}_{\sP_l}=\left(0,0,c_1\pi,\tau(c_1\pi),\dots, \tau(c_{a-1}\pi), c_a\pi\right)\in 
			\sO_{V(d_1)}\oplus \bigoplus_{i=1}^a\sO_{V(m_i)}\oplus \sO_{V(m_i')}. $$
			\item If $\sP_l$ is of type $T_0$ we let
				$$\bold{c}_{\sP_l}=\left(0,c_1\pi,\tau(c_1\pi),\dots, \tau(c_{a-1}\pi), c_a\pi,\tau(c_a\pi)\right)\in 
			\sO_{V(d_1)}\oplus \left(\bigoplus_{i=1}^a\sO_{V(m_i)}\oplus \sO_{V(m_i')}\right) \oplus \sO_{V(d_2)}.$$
	\end{enumerate}

	\begin{claim}\label{claim:correction_vector_is_invariant}
	Each vector $\bold{c}_{\sP_l}$ is $\tau$-invariant.
	\end{claim}	
	\begin{proof}\renewcommand{\qedsymbol}{$\lozenge$}
	It suffices to check $\tau_\eta$-invariance at the generic points. This is clear if $\sP_l$ is of type $T_0, T_1$ or $T_2$ (recall that, under the assumptions of \autoref{claim:nodes_are_split}, $\tau_\eta$ acts trivially on the residue field of a fixed point). If $\sP_l$ is of type $T_\circlearrowleft$, then we only have to check that $\tau_\eta(c_a\pi)=0$. For simplicity, identify the residue field of every point of $\sP_l$ with $k(\eta)$. Then by construction we have
		\begin{eqnarray*}
		\res_{m_1}(v)&=&\res_{m_1'}(v)+c_1\pi \\
		& = &\res_{m_2}(v)+\tau_\eta(c_1\pi)\\
		&=&\dots \\
		&=&\res_{m_{a-1}'}(v)+c_{a-1}\pi\\
		&=&\res_{m_a}(v)+\tau_\eta(c_{a-1}\pi).
		\end{eqnarray*}
	On the other hand, since $\tau(m_1)=m_{a}'$, by $\tau$-invariance of $v$ we have $\res_{m_1}(v)=\res_{m_a'}(v)$.
		In particular 
		$$\tau_\eta(c_{a-1}\pi)=\res_{m_a'}(v)-\res_{m_a}(v)=\gamma_a\pi,$$ 
		and by \autoref{eqn:def_correction_term} this implies that $c_a\pi=0$.
	\end{proof}
		
	We can finally conclude the proof of \autoref{thm:commutativity_in_slc_case}. Consider the concatenated vector
			$\bold{c}=\left(\bold{c}_{\sP_l}\right)_l$.
		Up to permutation of its entries, it belongs to $\bigoplus_{p\in\sP} \sO_{V(p)}$. By \autoref{claim:correction_vector_is_invariant} it is $\tau$-invariant. Therefore by \autoref{claim:nodes_are_disjoint_normal} and \autoref{lemma:equivariant_Chinese_remainder_thm}, there exists $c\in \sO_{\bar{D}}^+$ that reduces to $\bold{c}$. The element $v+c\pi$ is $\tau$-invariant and by \autoref{eqn:correction_term} it satisfies the descent condition of \autoref{corollary:descend_section_by_gluing}. This proves \autoref{claim:correction_term_for_descent}, and \autoref{thm:commutativity_in_slc_case} follows as already observed.
\end{proof}

\begin{remark}[Higher dimensions]\label{rmk:cond_for_comm_thm}
Let me say a word about a potential analogue of \autoref{thm_intro:commutativity_for_surfaces} in higher dimensions. Let $(X,\Delta+X_\pi)\to \Spec R$ be a relative slc pair of arbitrary dimension, with normalization $(\bar{X},\bar{D}+\bar{\Delta}+\bar{X}_\pi)$. There are two main inputs in our proof of \autoref{thm_intro:commutativity_for_surfaces}: the reducedness of the scheme-theoretic intersection $\bar{D}\cap \bar{X}_\pi$, and the game of liftings explained in \autoref{claim:correction_term_for_descent}.
	\begin{itemize}
		\item We have obtained the reducedness of $\bar{D}_\pi$ as a combination of properties of dlt pairs and vanishing theorems for birational morphisms. 
		\item The proof of \autoref{claim:correction_term_for_descent} is essentially a combinatorial game, once \autoref{claim:nodes_are_disjoint_normal} is established. In higher dimensions, we cannot expect the closures in $\bar{D}^n$ of the preimages of the nodes of $\bar{D}$ to be disjoint or normal. But as before, we it should be possible to reduce to dlt pairs: suppose $\varphi\colon (\bar{Y},E+\bar{D}_Y+\bar{\Delta}_Y+\bar{Y}_\pi)\to (\bar{X},\bar{D}+\bar{\Delta}+\bar{X}_\pi)$ is a crepant model with the properties that
			\begin{enumerate}
				\item the irreducible components of the preimages in $\bar{D}_Y^n$ of the nodes on $\bar{D}_Y$ are normal, and their union is semi-normal, and
				\item $\varphi_*\sO_{\bar{D}_Y}=\sO_{\bar{D}}$.
			\end{enumerate}
	Then our combinatorial game can be played on $\bar{D}_Y^n$ (the non-disjointness is not a problem thanks to properties of semi-normal schemes \cite[Lemma 10.19]{Kollar_Singularities_of_the_minimal_model_program}), and the relevant section can be pushed down to $\bar{D}$ by the second property. 
	\end{itemize}
Such a proof is available in characteristic $0$ (the reducedness condition holds by \cite[Corollary 3]{Kollar_Seminormal_log_center_and_deformations}, and the existence of a nice dlt crepant model by \cite[10.37, 4.20]{Kollar_Singularities_of_the_minimal_model_program}. This gives an alternative proof of a special case of  \cite[7.21]{Kollar_Singularities_of_the_minimal_model_program}. On the other hand, the validity of these auxiliary results in higher dimension in positive or mixed characteristic is unclear at the moment.
\end{remark}

\subsection{Serre properties of the fibers}\label{subsection:Serre_properties}

In this subsection we consider the Serre conditions $S_r$ on the special fiber $Y_\pi$ (we keep the notations of \autoref{notation:slc_family_II}). Since we always assume $Y\to \Spec R$ to be flat, $\pi$ is not a zero-divisor in $\sO_Y$ and so
		$$Y_\pi \text{ is }S_r \quad \Longleftrightarrow\quad Y \text{ is } S_{r+1} \text{ along }Y_\pi.$$
	The analog equivalence holds for $X_\pi\subset X$, since $X$ is automatically flat over $\Spec R$, by \cite[III.9.7]{Hartshorne_Algebraic_Geometry} and the fact that $p\colon X\to Y$ is finite.  Similarly, an analog equivalence holds for $D$ and $D/R_D(\tau)$ as soon as one of them (equivalently both of them) is flat over $\Spec R$.
	
	\begin{remark}
	By \cite[7.8.6.iii]{EGA_IV.2} the $S_i$ loci of excellent schemes are open. If $Y\to \Spec R$ is proper (or simply a closed morphism), a closed subset that is disjoint from $Y_\pi$ must be empty. In this case the previous equivalence becomes: $Y_\pi$ is $S_r$ if and only if $Y$ is $S_{r+1}$, and similarly for $X,D$ and $D/R_D(\tau)$ under suitable hypothesis.
	\end{remark}
	
	To approach the Serre properties of $Y$, we rely on the fact that it can be described as a universal push-out \cite[Proposition 25]{Kollar_Quotients_by_finite_equivalence_relations}. The main technical tool is the following lemma.

\begin{lemma}\label{lemma:homological_lemma}
Consider a commutative square of complexes of abelian groups
		$$\begin{tikzcd}
		A^\bullet \arrow[r, "\alpha^\bullet"] \arrow[d, "\beta^\bullet"] & B^\bullet \arrow[d, "q^\bullet"] \\
		C^\bullet \arrow[r, "p^\bullet"] & D^\bullet 
		\end{tikzcd}$$
which is a pullback diagram at every degree. Then for every $j$ there exists a natural map
		$$\xi^j\colon h^j(A^\bullet )\to h^j(B^\bullet )\times_{h^j(D^\bullet)}h^j(C^\bullet )$$
	which is surjective. If $h^{j-1}(D^\bullet)=0$ and $p^{j-2}$ is surjective, then $\xi^j$ is bijective.
\end{lemma}
\begin{proof}
Let us embark for some diagram-chasing in
		$$\begin{tikzcd}
		&&& A^{j-1}\arrow[d] \arrow[ddlll]\arrow[ddrrr]&&&& \\
		&&& A^j\arrow[dll, "\alpha^j" above left]\arrow[drr, "\beta^j"] &&&& \\
		B^{j-1}\arrow[ddrrr, "q^{j-1}" below left]\arrow[r] & B^j\arrow[drr]&&&& C^j\arrow[dll] & C^{j-1}\arrow[ddlll]\arrow[l] & C^{j-2}\arrow[dddllll, "p^{j-2}"]\arrow[l]\\
		&&& D^{j} &&&& \\
		&&& D^{j-1}\arrow[u] &&&& \\
		&&& D^{j-2}\arrow[u] &&&&
		\end{tikzcd}$$
	We denote the differential maps of $A^\bullet$ by $d_A^j\colon A^j\to A^{j+1}$, and similarly for the other complexes.
	
If $x\in \ker(d^j_A)$, then clearly $\alpha^j(x)\in\ker(d^j_B)$ and $\beta^j(x)\in \ker(d^j_C)$. This defines a map
		$$\xi^j_0\colon \ker(d^j_A)\to \ker(d^j_B)\times_{\ker(d^j_D)}\ker(d^j_C).$$
	We claim that $\xi^j_0$ is surjective. Indeed, a pair $(s,t)\in \ker(d^j_B)\times_{\ker(d^j_D)}\ker(d^j_C)$ gives, by the pullback property, a unique element $x\in A^j$ such that $\alpha^j(x)=s$ and $\beta^j(x)=t$. Since 
			$$\alpha^{j+1}(d_A^jx)=d^j_B(\alpha^j(x))=0,\quad \beta^{j+1}(d_A^jx)=d_B^j(\beta^j(x))=0$$
	the pullback property ensures that $d^j_A(x)=0$. Thus $\xi^j_0$ is surjective.
	
	It is easy to see that $\xi_j'$ descends to a surjective map
			$$\xi^j \colon h^j(A^\bullet)\to h^j(B^\bullet)\times_{h^j(D^\bullet)}h^j(C^\bullet).$$
	Now let us discuss the injectivity of $\xi^j$. Let $x\in \ker(d^j_A)$ be such that $\alpha^j(x)=d^{j-1}_B(y)$ and $\beta^j(x)=d^{j-1}_C(z)$ for some $y\in B^{j-1}$ and $z\in C^{j-1}$. We investigate whether $x\in \im(d^{j-1}_A)$. The element
		$\partial:=q^{j-1}(y)-p^{j-1}(z)$
is non-zero in general, but it belongs to $\ker(d^{j-1}_D)$. Thus if $h^{j-1}(D^\bullet)=0$, there exists $\partial'\in D^{j-2}$ such that $d^{j-2}_D(\partial')=\partial$. If $p^{j-2}$ is surjective, there exists $\delta'\in C^{j-2}$ such that $p^{j-2}(\delta')=\partial'$. Set $z_1:=z+d^{j-2}_C(\delta')$. We have
		$$p^{j-1}(z_1)=p^{j-1}(z)+\delta=q^{j-1}(y),$$
thus by the pullback property the pair $(y,z_1)$ corresponds to a unique $x'\in A^{j-1}$. Since $\alpha^j(d^{j-1}_A(x'))=d^{j-1}_B(y)=\alpha^j(x)$ and $\beta^j(d^{j-1}_A(x'))=d_C^{j-1}(z_1)=d_C^{j-1}(z)=\beta^j(x)$, the pullback property shows that $x=d_A^{j-1}(x')$. This concludes the proof.
\end{proof}

\begin{lemma}\label{lemma:local_cohomology_on_semi_local_ring}
Let $A$ be a Noetherian semi-local ring with Jacobson ideal $I=\fm_1\cap \dots\cap \fm_n$. Let $M$ be an $A$-module. Then $H^r_I(M)=0$ if and only if $H^r_{\fm_iA_{\fm_i}}(M_{\fm_i})=0$ for every $\fm_i$.
\end{lemma}
\begin{proof}
Since the localization maps $A\to A_{\fm_i}$ are flat, we have isomorphism s
		$$H^r_I(M)\otimes_AA_{\fm_i}\cong H^r_{IA_{\fm_i}}(M_{\fm_i})\quad \text{for every }i.$$
Since the $\fm_i$ are pairwise coprime, we have $I=\fm_1\cdots\fm_n$ and therefore $IA_{\fm_i}=\fm_iA_{\fm_i}$. Thus the result follows from the fact that a module is trivial if and only if its localization at every maximal ideal is trivial.
\end{proof}

\begin{proposition}\label{proposition:Serre_property_of_quotient}
Let $(X,D,\tau)\to Y$ be as in \autoref{notation:slc_family_II} (flatness of $Y\to \Spec R$ is not necessary). Assume in addition that $D$ is $S_r$.
	\begin{enumerate}
		\item If $X$ and $D/R_D(\tau)$ are $S_{r+1}$, then $Y$ is $S_{r+1}$.
		\item If $D$ is $S_{r+1}$, then the converse implication in $(a)$ holds.
	\end{enumerate}
\end{proposition}

\begin{remark}
As the proof will show, the statement of \autoref{proposition:Serre_property_of_quotient} is local on $Y$. That is, the statement holds if we replace adequately \textit{being $S_i$} with \textit{being $S_i$ at $y\in Y$}, respectively with \textit{being $S_i$ along $p^{-1}(y)\subset X$}.
\end{remark}

\begin{proof}
For simplicity, we write $\sW=D/R_D(\tau)$. Then by \cite[Proposition 25]{Kollar_Quotients_by_finite_equivalence_relations} the diagram
		$$\begin{tikzcd}
		D\arrow[r, hook] \arrow[d] & X\arrow[d, "p"] \\
		\sW\arrow[r, hook] & Y
		\end{tikzcd}$$
is a universal push-out and $\sW\hookrightarrow Y$ is a closed embedding. Thus if $y\in Y$ is a point with $\fm_{y}=(f_1,\dots,f_n)\subset \sO_y:=\sO_{Y,y}$, then the diagrams
		\begin{equation}\label{eqn:push_out_diagram}
		\begin{tikzcd}
		(\sO_{D\times \sO_{y}})_{f_{i_0}\cdots f_{i_s}} && (\sO_{X\times \sO_{y}})_{f_{i_0}\cdots f_{i_s}} \arrow[ll] \\
		(\sO_{\sW,y})_{f_{i_0}\cdots f_{i_s}} \arrow[u] && (\sO_{y})_{f_{i_0}\cdots f_{i_s}}\arrow[ll, "p^{i_0,\dots,i_s}" above]\arrow[u]
		\end{tikzcd}
		\end{equation}
are pullback diagrams of semi-local rings. Notice that $p^{i_0,\dots,i_s}$ is always surjective, since $\sW\hookrightarrow Y$ is a closed embedding. 

For $\sA\in\{\sO_{y},\sO_{\sW,y},\sO_{X\times\sO_{y}},\sO_{D\times\sO_{y}}\}$ we denote by $\check{C}^\bullet(\bold{f};\sA)$ the alternating \v{C}ech complex associated to $\sA$ and to $\bold{f}=\{f_1,\dots,f_n\}$. Recall that
		$$\check{C}^\bullet(\bold{f};\sA)=\left[0\to \sA\to \bigoplus_{i}\sA_{f_i}\to \bigoplus_{i_0<i_1}\sA_{f_{i_0}f_{i_1}} \to \dots \to \sA_{f_1\cdots f_n}\to 0\right]$$
with differential defined as alternated sums of the natural localization maps. Thus it is easy to see that the diagrams \autoref{eqn:push_out_diagram} induce a commutative square of complexes
		$$\begin{tikzcd}
		\check{C}^\bullet(\bold{f};\sO_{D\times \sO_y}) & \check{C}^\bullet(\bold{f};\sO_{X\times\sO_y})\arrow[l] \\
		\check{C}^\bullet(\bold{f};\sO_{W,y}) \arrow[u] & \check{C}^\bullet(\bold{f};\sO_{y})\arrow[u]\arrow[l, "p^\bullet" above, twoheadrightarrow]
		\end{tikzcd}$$
where $p^\bullet$ is surjective and the square is a pullback at every degree. Now the cohomology of $\check{C}(\bold{f};\sA)$ is equal to the local cohomology of $\sA$ along $\fm_y\sA$ \cite[Theorem 7.13]{24_hours_of_cohomology}. By the first part of \autoref{lemma:homological_lemma}, we get a surjective map
		$$H^j_{y}(\sO_{y})\twoheadrightarrow H^j_{X_y}(\sO_{X\times\sO_{y}})\times_{H^j_{D_y}(\sO_{D\times\sO_{y}})}H^j_{y}(\sO_{W,y}).$$
		Now assume that $D$ is $S_r$. Then $H^j_{z}(\sO_{D,z})=0$ for every $j<r$ and $z\in D$. If $z$ maps to $y\in Y$, then $\sO_{D,z}$ is the localization of $\sO_{D\times \sO_y}$ at some maximal ideal. Moreover $\fm_y\sO_{D,z}$ is equal (up to taking its radical, to which the local cohomology is insensitive) to the Jacobson radical of $\sO_{D\times\sO_y}$. Thus $H^j_{D_y}(\sO_{D\times\sO_{y}})=0$ for every $j<r$ by \autoref{lemma:local_cohomology_on_semi_local_ring}. Therefore by the second part of \autoref{lemma:homological_lemma}, we have an isomorphism
				\begin{equation}\label{eqn:equality_local_cohomology}
				H^j_{y}(\sO_{y})\cong H^j_{X_y}(\sO_{X\times\sO_{y}})\times H^j_{y}(\sO_{W,y}) \quad \forall j< r.
				\end{equation}
		Similarly to above the group $H^j_{X_y}(\sO_{X\times \sO_y})$ is equal to the $j^\text{th}$ local cohomology group of the semi-local ring $\sO_{X\times\sO_y}$ along its Jacobson radical. 
		
		Thus we see that if $X$ and $\sW$ are $S_{r+1}$, then by \autoref{eqn:equality_local_cohomology} we have $H^j_y(\sO_y)=0$ for every $j<r$. This proves the first point. Conversely, if $D$ and $Y$ are $S_{r+1}$, then the isomorphisms \autoref{eqn:equality_local_cohomology} hold for all $j\leq r$, and thus $X$ and $W$ are $S_{r+1}$.
\end{proof}

\bibliographystyle{alpha}
\bibliography{Bibliography}

\end{document}